\documentclass[12pt,oneside]{amsart}
\usepackage[a4paper]{geometry}
\usepackage{amsfonts,amsmath,amscd}
\usepackage{epsfig,graphics}
\usepackage{amssymb}
\usepackage[foot]{amsaddr}
\numberwithin{equation}{section}
\usepackage{amscd}
\usepackage{hyperref}
\usepackage{bbm}%
\usepackage{cmupint}
\usepackage{xcolor}%

\usepackage{tikz-cd}%

\usepackage{enumitem}
\allowdisplaybreaks
\usepackage{accents} %
\newcommand*{\dt}[1]{%
	\accentset{\mbox{\large\bfseries .}}{#1}}

\newtheorem{conjecture}{Conjecture}

\newcommand{\BC}{{\mathbb{C}}}

\newcommand{\gO}{\Omega}

\newcommand{\gl}{\lambda}
\newcommand{\ga}{\alpha}

\def\a{\aleph}

\newcommand{\ol}[1]{\overline{#1}}

\newcommand{\re}{\text{Re}}

\DeclareMathOperator{\tprob}{\mathrm{Prob}}

\newcommand{\diam}{\text{diam}}

\DeclareMathOperator{\Mob}{M\mathnormal{\ddot{\mathrm{o}}}b}
\DeclareMathOperator{\Lip}{\mathrm{Lip}}
\def\D#1{|#1'|}
\def\Db#1{|(#1)'|}
\def\C{\mathcal{C}}
\def\constone{\mathbf{1}}

\newtheorem{prop}{Proposition}[section]
\newtheorem*{prop*}{Proposition}
\newtheorem{thm}{Theorem}

\newtheorem*{thm*}{Theorem}
\newtheorem{lem}[prop]{Lemma}

\newtheorem*{cor*}{Corollary}

\theoremstyle{definition}

\newtheorem{defn}[prop]{Definition}
\newtheorem*{defn*}{Definition}
\newtheorem{rem}[prop]{Remark}
\newtheorem*{rem*}{Remark}
\newtheorem{exam}[prop]{Example}

\newtheorem{calc}[prop]{Calculation}

\title{Sobolev spaces and uniform boundary representations}
\author{Kevin Boucher}
\author{J\'{a}n \v{S}pakula}
\address{School of Mathematical Sciences, University of Southampton, Highfield, Southampton, SO17 1BJ, UK}
\email{kevin.boucher01@gmail.com, jan.spakula@soton.ac.uk}
\date{\today}

\begin{document}

\begin{abstract}
 We prove uniform boundedness of certain boundary representations on appropriate  fractional Sobolev spaces $W^{s,p}$ with $p>1$ for arbitrary Gromov hyperbolic groups. These are closed subspaces of $L^p$ and in particular Hilbert spaces in the case $p=2$.
 
This construction allows us, for an appropriate choice of $p$, to approximate the trivial representation through uniformly bounded representations.
This phenomenon does not have analogue in the setting of isometric representations whenever the hyperbolic group considered has the Property (T) \cite{Bader-Furman-Gelander-Monod:2007}.

The key is the introduction of a notion of metrically conformal operator on a metric space endowed with a conformal structure \`{a} la Mineyev \cite{Mineyev:2007} and a metric analogue of the isomorphisms of Sobolev spaces induced by the Cayley transform \cite{Astengo-Cowling-DiBlasio:2004}.
\end{abstract}

\maketitle

\section{Introduction}\label{sect:introduction}

\subsection{Landscape}\label{subsect:intro-landscape}

The theory of spherical representations deals with representations realised as spaces of ordinary functions over the flag manifold of semi-simple Lie groups.
In the broader context of geometric groups such as Gromov hyperbolic groups, the flag manifold is substituted by a certain type of boundary. We refer to these representations as  \textit{boundary representations}. This terminology was popularised by Bader and Muchnik in \cite{Bader-Muchnik:2011}, who investigated the irreducibility of the Koopman representation of a negatively curved group on its visual boundary.

Suppose that a group $\Gamma$ acts isometrically on a negatively curved space $X$, and thus on its boundary $\partial X$ by quasi-M\"obius transformations. The boundary representations are then ones given by the formula
\begin{equation}\label{eq:pi_z}
\pi_z(g)\phi = \left(\tfrac{\mathrm{d}g_*\nu}{\mathrm{d}\nu}\right)^z \cdot g^*\phi,
\end{equation}
where $z\in S=\{w\in\BC:0\le\mathrm{Re}(w)\le1\}$, $\nu$ is an appropriate measure on $\partial X$, and $\phi\in\mathcal{D}_z$ for some vector space $\mathcal{D}_z$ of functions on $\partial X$. We denote by $g^*$ the induced action on functions, i.e.~$(g^*\phi)(\xi)=\phi(g^{-1}\xi)$.
In this context, two natural categories stand out: first, the \emph{principal series}, when $z=\frac{1}{2}+it$, $t\in\mathbb{R}$, on $\mathcal{D}_z=L^2(\partial X, \nu)$; these are unitary. Second, the \emph{complementary series}, for $z=s\in\mathbb{R}\cap S$, which are our subject of interest. 
These do not have obvious unitary structures, that is, a domain $\mathcal{D}_{s}$ with a Hilbert space norm making $\pi_{s}$ unitary.

These two categories of representations were extensively studied and play an important in harmonic analysis of semi-simple Lie groups \cite{MR163988, MR855239, MR754767}.
One of the motivations is the strong analogy between rank 1 lattices and hyperbolic groups, and thus a potential extension of the theory to this broader class.
This was achieved for free groups by Fig\`a-Talamanca and Piccardello, who were able to develop an intrinsic theory of spherical representations in \cite{Figa-Talamanca-Picardello:book:1983}.

Little is known about the role of principal series and their properties for general hyperbolic groups, and it is only recently that complementary series have been investigated, cf. \cite{Boucher-complementary:2020,Boucher-special:2020,Boyer-Picaud:2023}.
These representations, when available, describe a path between a tempered representation (when $z=\frac{1}{2}$) and the trivial representation (when $z$ approaches $0$ or $1$).
As such, Kazhdan's Property (T) is a natural obstruction for their existence for all $z\in[\frac{1}{2},1)$, and the largest admissible parameter, $z$ in this interval, as a sort of distance from the trivial representation. 

When we widen the context from unitary representations to \emph{uniformly bounded} representations, there is no such obvious restriction. The purpose of this paper is to construct uniform structures\footnote{If $\pi$ is a linear representation of a group $G$ on a complex vector space $V$, a \emph{uniform structure} is a norm on $V$, such that the completion is a Hilbert space $H$, and such that $\pi$ extends to a uniformly bounded representation on $H$, i.e.~$\sup_{g\in G} \|\pi(g)\|<\infty$.} for these representations regardless of Property (T). This is motivated by the following conjectures:
\begin{conjecture}[Spectral conjecture]\label{conj:Shalom2}
    Let $\Gamma$ be a hyperbolic group. There exists $f:\Gamma\to\mathbb{R}_+$, such that for every compact set $K\subset\Gamma$ and $\varepsilon>0$ there exists a uniformly bounded representation $\pi_{K,\varepsilon}$ of $\Gamma$ on a Hilbert space $H_{K,\varepsilon}$ without non-zero invariant vectors,  and a unit vector $v_{K,\varepsilon}\in H_{K,\varepsilon}$ such that
    \begin{equation*}
    \diam(\pi_{K,\varepsilon}(K)v_{K,\epsilon})\leq \varepsilon,
    \end{equation*}
    and for all $g\in\Gamma$ $\|\pi_{K,\varepsilon}(g)\|\le f(g)$ \footnote{
    	The last condition on pointwise uniform boundedness over $K\subset\Gamma$ and $\varepsilon>0$ implies that one can apply a limiting procedure (e.g.~a direct sum, or an ultraproduct) to a family of $(\pi_{K,\varepsilon})_{K,\varepsilon}$ to obtain a single linear representation $\pi$ of $\Gamma$ on a Hilbert space which does not have non-zero invariant vectors, but almost has invariant vectors (in the language of \cite{Bader-Furman-Gelander-Monod:2007}).}.
\end{conjecture}

For infinite groups with Property (T), to the best of our knowledge the Conjecture is only known to be true for $\mathrm{Sp}(n,1)$, $n\ge1$, due to the work of Cowling \cite{Cowling:1982, Cowling:1983} using boundary representations.
See also \cite{Knapp-Stein:1971, Boucher-complementary:2020, Boucher-special:2020} for explicit boundary constructions.
Note that higher rank lattices do not satisfy this conjecture \cite{delaSalle:2019} even when the class of Banach spaces is extended to $L^p$ (compare with Theorem \ref{thm:almost-inv-subspace-Lp-in-intro}).
\begin{rem}
    For a uniformly bounded representation $\pi$ of a group $\Gamma$, denote $\|\pi\|=\sup_{g\in \Gamma}\|\pi(g)\|$. One may also consider a stricter variant of Conjecture 1, where one would require a uniform bound (depending only on $\Gamma$) on $\|\pi_{K,\varepsilon}\|$. Dooley's work \cite{Dooley:2004} implies also this stronger statement for $Sp(n,1)$.
\end{rem}
Conjecture \ref{conj:Shalom2} is closely related to another one:
\begin{conjecture}[Attributed to Shalom, \cite{Shalom:2001}]\label{conj:Shalom1}
	For every hyperbolic group there exists a uniformly bounded representation on a Hilbert space which admits a proper cocycle. Equivalently, every hyperbolic group admits a uniformly Lipschitz proper action on a Hilbert space by affine transformations.
\end{conjecture}
Indeed, in the context of unitary representations, Schoenberg's Theorem \cite[Prop.~2.11.1]{Bekka-delaHarpe-Valette:2008} explains that there is a correspondence between semi-groups of unitary representations approaching the trivial one, and proper cocycles with unitary coefficients, i.e.~proper isometric actions on Hilbert spaces.
As this correspondence no longer exists in the uniformly bounded context, the two conjectures might not be equivalent.

Clearly, a-T-menable hyperbolic groups (for example surface groups, and generally hyperbolic groups acting geometrically on CAT(0) cube complexes) satisfy Conjecture \ref{conj:Shalom1}. Remarkably, this Conjecture has been recently proved also for $\mathrm{Sp}(n,1)$, $n\ge1$, by S.~Nishikawa \cite{Nishikawa:2020}.

Another motivation to construct and study analogues of complementary series for discrete groups comes from the research surrounding the Baum--Connes conjecture \cite{Baum-Connes:2000,GomezAparicio-Julg-Valette:2019}. Roughly speaking, one of the ingredients in the Dirac-dual Dirac method \cite{Kasparov:1988,Higson-Kasparov:2001} for proving the Conjecture is a path of representations that connect the regular with the trivial representation.
The purely C*-algebraic methods ask for such a path through unitary representations and Kazhdan's Property (T) prevents its existence. There are two existing approaches to bypass this problem. The first, pioneered and remarkably successfully executed by Lafforgue \cite{Lafforgue:2002,Lafforgue:2012}, is to step outside of the C*-world, develop and use Banach KK-theory --- this allows to use a path through `slowly growing representations'. The second, for $\mathrm{Sp}(n,1)$ by Julg \cite{Julg:2002,Julg:2016,Julg:2019}, uses uniform structures for the complementary series representations 
\cite{Cowling:1982,Cowling:1983}.

Uniform structures on the representations $\pi_s$ with $s\in [0,1]$ introduced above appear to be related to positive conformal  operators and their corresponding Sobolev spaces. It is known that those operators spectrum are involved into important functional analytic inequality such as Lieb, Sobolev and others
\cite{Branson:1995, MR717827}. Our approach stress further this connection in a purely metric setting and might be a source of interest beyond our prime concern.

\subsection{Results}\label{subsect:intro-results}

We work in the setting of metric measure spaces. (The reader may want to keep in mind an example of a Gromov boundary of a non-e\-le\-men\-ta\-ry hyperbolic group or $\mathbb{R}^n$.)
A \emph{metric measure space} is a triple $(Z,d,\nu)$, such that $(Z,d)$ is a complete metric space without isolated points, and $\nu$ is a Borel measure
on $Z$. We always assume that it is \emph{Ahlfors $D$-regular} for some $D\geq 0$, i.e.
$$\nu(B(z,r))\asymp r^D\quad\text{ for $z\in Z$ and $0\leq r< \diam(Z)$.}$$
In this case, $D$ is necessarily the Hausdorff dimension of $(Z,d)$ and $\nu$ is equivalent to the Hausdorff $D$-measure \cite[Section 1.4.3]{Mackay-Tyson:2010}. 
We also need to consider Ahlfors $D$-regular quasi-metric measure spaces, i.e.~triples $(Z,d,\nu)$ as above where $d$ is only a \emph{quasi-metric}\footnote{Alternatively, a \emph{quasi-metric} is a~function satisfying the axioms of a metric with the exception of the triangle inequality being relaxed to $d(x,y)\leq K(d(x,z)+d(z,y))$ for some $K\geq1$. However for the application to hyperbolic groups it is sufficient to stick to genuine metrics throughout.}, i.e.~a symmetric function $d:Z^2\to[0,\infty)$ for which there exists a genuine metric $\rho\asymp d$.

For an Ahlfors $D$-regular quasi-metric measure space $(Z,d,\nu)$, $p,s\geq0$, and a measurable function $\phi: Z\to\mathbb{C}$, we denote
\begin{equation}\label{eq:defn-of-[]sp-intro}
	[\phi]_{s,p}^p=_\text{def}\iint_{Z\times Z}\frac{|\phi(\xi)-\phi(\eta)|^p}{d^{D+sp}(\xi,\eta)}\mathrm{d}\nu(\xi)\mathrm{d}\nu(\eta).
\end{equation}
We define the \emph{fractional $(s,p)$-Sobolev space} $W^{p,s}(Z)$ to be
\begin{equation*}
	W^{s,p}(Z)=\Big\{\phi\in L^p(Z) : [\phi]_{s,p} < \infty\Big\},
\end{equation*}
with the norm 
$\|\phi|W^{s,p}\|^p=\|\phi | L^p\|^p+[\phi]_{s,p}^p$ for $\phi\in W^{s,p}(Z)$.
The \emph{homogeneous fractional Sobolev space} $\dt{W}^{s,p}(Z)$ is defined as the completion of 
\begin{equation*}
	\{\phi:Z\to\mathbb{C} : \phi\text{ is compactly supported, measurable and $[\phi]_{s,p}^p$ finite }\}
\end{equation*}
with the semi-norm $[\cdot]_{s,p}$; see also Section \ref{sect:sobolev-spaces}. Note that when $p=2$, these spaces are Hilbert spaces and when $Z=\mathbb{R}^n$ we obtain the classical fractional Sobolev spaces \cite{Sobolev-Hitchhikers}.

Given a complete metric space $(Z,d)$  and $a\in Z$, we denote $Z_a=Z\setminus\{a\}$ and
\begin{equation}\label{eq:d_a-defn-in-intro}
	d_a(\xi,\eta) =_{\text{def}} \frac{d(\xi,\eta)}{d(\xi,a)d(\eta,a)}
\end{equation}
on $Z_a$.
 A \emph{Cayley transform of $(Z,d)$ at $a$} is any metric space $(Z_a,\rho_a)$, such that $\rho_a$ induces the subspace topology on $Z_a\subset Z$ and $\rho_a\asymp d_a$.
This generalises the concept of a stereographic projection within the context of metric spaces, cf.~\cite[Section 9]{Mineyev:2007}. 
By \cite[Lemma 4.5]{Haissinsky:2015}, one can always find a metric $\rho_a$ bi-Lipschitz to $d_a$, with Lipschitz constants independent of $a\in Z$. Moreover, $d_a$ is already a metric whenever $(Z,d)$ is a Gromov boundary of a strongly hyperbolic metric space\footnote{Any finitely generated Gromov hyperbolic group admits a geometric action on strongly hyperbolic metric space, see \cite{Mineyev-Yu:2002,Nica-Spakula:2016}.} \cite[Theorem 16]{Mineyev:2007}.
(Mineyev refers to the collection $\{(Z_a,d_a)\}_{a\in Z}$ as a ``metric conformal structure''.)
As detailed later, Cayley transform preserves the Hausdorff dimension, and any choice of a metric bi-Lipschitz to $d_a$ leads to equivalent fractional Sobolev spaces.

Our main result is the following:
\begin{thm}\label{thm:Cayley-in-intro}
    Let $(Z,d,\nu)$ be a bounded Ahlfors $D$-regular metric measure space, $a\in Z$, $t\in \mathbb{R}$, $s\in(0,1)$ and $p\in [1,\infty)$ with $sp< D$. Then the map
	\begin{equation}\label{eq:Omega-z-p-defn-in-intro}
	    \Omega_{s+it,p}(a):W^{s,p}(Z)\rightarrow \dt{W}^{s,p}(Z_a);
    \quad\psi\mapsto d^{\frac{2D}{p}-2s-2it}(a,\cdot)\psi|_{Z_a} 
    \end{equation}
    is an isomorphism, with the inverse given by
    \begin{equation}\label{eq:Omega-z-p-inverse-defn-in-intro}
    	\Omega^{-1}_{s+it,p}(a):\dt{W}^{s,p}(Z_a)\rightarrow W^{s,p}(Z);
    \quad\psi\mapsto [d^{-\frac{2D}{p}+2s+2it}(a,\cdot)\psi].
    \end{equation}
    where $[f]$ stands for the class of the $f$ in $L^p(Z)$.
    Moreover, the norms of $\gO_{s+it,p}(a)^{\pm1}$ are bounded independently of $a\in Z$ and $t\in \mathbb{R}$.
\end{thm}
This result may be compared with the work of Astengo, Cowling, and DiBlasio \cite{Astengo-Cowling-DiBlasio:2004} for $\mathrm{Sp}(n,1)$, $n\ge1$, as we obtain an analogous result in the general setting of metric spaces. However the range of the parameter $s$ is restricted to $(0,1)$. The analogy goes further in that the isomorphisms $\Omega_{\cdot}$ come from the Radon--Nikodym derivative of the canonical map $(Z_a,d_a,\nu_a) \hookrightarrow (Z,d,\nu)$ which is --- in our setting --- the inclusion.

In the same spirit as in \cite{Astengo-Cowling-DiBlasio:2004}, we use  Theorem~\ref{thm:Cayley-in-intro} to construct uniformly bounded representations:\\
A self-homeomorphism of a metric space $(Z,d)$ is called \emph{M\"obius}, if it preserves cross-ratio, i.e. the quantity $[x,y;z,w]=_\text{def}\frac{d(x,z)d(y,w)}{d(x,w)d(y,z)}$ for pairwise distinct $x,y,z,w\in Z$.
We denote $\Mob(Z)$ the locally compact, second countable group of M\"obius transformations of $(Z,d)$. Any $g\in\Mob(Z)$ admits a \emph{metric derivative} $\D{g}\in\Lip_{+}(Z)$ (cf.~Preliminaries below).
If $\nu$ is the $D$-Hausdorff measure on $(Z,d)$, then
$[\mathrm{d}{g_*\nu}/\mathrm{d}\nu](\xi) 
= \Db{{g^{-1}}}^D(\xi)$ 
for $\xi\in Z$, see e.g.~\cite[Section 4]{Nica:2013}.

\begin{thm}\label{thm:UB-reps-in-intro}
Let $(Z,d,\nu)$ be an Ahlfors $D$-regular metric measure space where $\nu$ is the Hausdorff $D$-measure on $(Z,d)$. Let $t\in \mathbb{R}$, $s\in(0,1)$ and $p\in [1,\infty)$ with $sp< D$.
For $g\in\Mob(Z)$, $\psi\in W^{s,p}(Z)$, denote\footnote{The meaning for the parameter $s+it$ in this Theorem is at odds with the notation `$z$' used in the original formula \eqref{eq:pi_z}. Representation theory references often use \eqref{eq:pi_z}, but for considerations in this paper, the notation here is more natural. The two are related by $s+it=\frac{2D}{p}(\frac{1}{2}-z)$.}
\begin{equation}\label{eq:pi-p-formula-from-intro}
   	\pi^{(p)}_{s+it}(g)\psi=\Db{g^{-1}}^{\frac{D}{p}-s-it}\cdot g^*\psi.
\end{equation}
Then the representations $\pi^{(p)}_{s+it}$ of $\Mob(Z)$ on $W^{s,p}(Z)$ 
    are well defined and uniformly bounded, independently of $t\in\mathbb{R}$.
\end{thm}

The above Theorem yields uniformly bounded boundary representations of Gromov hyperbolic groups:
\begin{exam}\label{expl:hyperbolic-group-setup-in-intro}
	Let  $\Gamma$ be a non-elementary Gromov hyperbolic group.
	Then $\Gamma$ acts geometrically on a proper, roughly geodesic, strongly hyperbolic space $(X,\rho)$. \cite{Mineyev-Yu:2002, Nica-Spakula:2016}.
	Fix a basepoint $o\in X$. The Gromov boundary, $Z=\partial X$, can be endowed with a visual metric $d$, such that $\Gamma$ acts on $(Z,d)$ by M\"obius transformations \cite[Section 5]{Mineyev:2007}. Furthermore, $(Z,d)$ has finite Hausdorff dimension $D$. With the $D$-dimensional Hausdorff measure $\nu$, $(Z,d,\nu)$ is an Ahlfors $D$-regular metric measure space \cite{Coornaert:1993}. Hence the above Theorem yields uniform structures for the boundary representations $\pi_{s+it}^{(p)}$ of $\Gamma$ for $s\in(0,1)$, $sp<D$, $t\in\mathbb{R}$. Note that this corresponds to $\max\!\left(0,\frac{1}{2}-\frac{1}{D}\right)<\re(z)<\frac{1}{2}$ when $p=2$ in the notation of \eqref{eq:pi_z}.
\end{exam}
In particular for $p=2$, we obtain a Hilbertian uniform structure for $(\pi_s)_s$ with $s\in (0,1)$ and $2s\le D$, namely $W^{s,2}(Z)$.
To the best of our knowledge it is the first result that constructs uniformly bounded representations which applies to arbitrary hyperbolic groups, including those with the Property (T). Note that if $D\leq 2$, then such $\Gamma$ is necessarily a-T-menable, as in this case $\Gamma$ acts properly by isometries on an $L^2$-space by \cite{Nica:2013}. 
The infimum of all the constants $D\geq0$ that can appear in Example \ref{expl:hyperbolic-group-setup-in-intro} is essentially Mineyev's ``hyperbolic dimension'' \cite[Section 10]{Mineyev:2007}.

Our method requires a choice of  metric conformal operators (cf.~Example \ref{expl:fractional-Laplace}), and the ones we make in this paper has a limitation in their orders (see Section \ref{sect:sobolev-spaces}).
However, we believe that our approach will prove to be useful beyond this work (cf.~also Appendix \ref{appx:groupoids}) to approach Conjectures \ref{conj:Shalom2} and \ref{conj:Shalom1}.

\medskip

Let us outline another consequence of our work.
Property (T) does not prevent existence of a proper action on a $L^p$-spaces in general: all hyperbolic groups admit a proper isometric action on an $L^p$-space for $p\geq2$ large enough \cite{Yu:2005,Nica:2013,Bourdon:2016}. On the other hand, Bader, Furman, Gelander, and Monod \cite{Bader-Furman-Gelander-Monod:2007} proved that any locally compact second countable group $G$ with Property (T) also satisfies the following rigidity statement: any linear isometric representation $\rho$ of $G$ on a closed subspace of some $L^p$-space for $p\not=4,6,8,\dots$ without non-zero invariant vectors does not almost have invariant vectors, i.e.~for some compact subset $K\subset G$ one has $\inf_{\|v\|=1}\diam(\rho(K)v)>0$. In other words, for Property (T) groups, the trivial representation is isolated even among isometric representations on subspaces of $L^p$-spaces. We show that for hyperbolic groups, regardless of (T), one can approach the trivial representation through uniformly bounded representations on closed subspaces of~$L^p$:\\

\begin{thm}\label{thm:almost-inv-subspace-Lp-in-intro}
	Let $\Gamma$ be a non-elementary hyperbolic group, let $K\subset \Gamma$ be compact, and let $(Z,d,\nu)$ be as in Example \ref{expl:hyperbolic-group-setup-in-intro}. Then for all $p\in(D,\infty)$ and $s\in(0,\frac{D}{p})$, the uniformly bounded representations $\pi^{(p)}_s$ on $W^{s,p}(Z)$ have no non-zero invariant vectors, satisfy $\sup_{s\in(0,\frac{D}{p})}\|\pi_s^{(p)}(g)\|<\infty$ and 
	\begin{equation*}
		\diam(\pi_s^{(p)}(K){\bf1})\xrightarrow{s\rightarrow \frac{D}{p}} 0.
	\end{equation*}
\end{thm}
It remains to observe that $W^{s,p}(Z)$ is a closed subspace of an $L^p$-space (Proposition \ref{prop:Wsp-is-a-closed-subsp-of-L^p}), and the constant function $\constone$ is a unit vector in $W^{s,p}(Z)$ for all $p\in[1,\infty)$ and $s\in(0,1)$.

This result can be interpreted as a positive answer to an $L^p$-version of Conjecture \ref{conj:Shalom2}.

\medskip

We conclude this Section with a description of the dual representations in the Hilbert space case, i.e.~for the rest of this Section we assume that $p=2$, and denote $H^s=W^{s,2}$. We denote the dual to the space $H^s$ by $H^{-s}$. 
Similarly as for the free groups and in the smooth case, the norm on $H^{-s}(Z)$ can be expressed in terms of potentials, %
using potential theory on metric measure spaces \cite{Grigoryan-Hu-Hu:2018,Hu-Zahle:2009, Chen-Kumagai:2003}.

Before stating the result, recall that the Knapp--Stein operator
\begin{equation}\label{eq:Knapp-Stein-in-intro}
	I_s(\phi)(\eta)=\int_Z
	\frac{\phi(\xi)}{d^{D-2s}(\xi,\eta)}\mathrm{d}\nu(\xi)
\end{equation}
was studied in \cite{Boucher-complementary:2020} to describe a unitary structure for $\pi_{-s}^{(2)}$ when it is positive definite. This is the case for example when $(Z,d)$ arises as a boundary of a roughly geodesic strongly hyperbolic metric space $(X,\rho)$ where $\rho$ is conditionally negative definite, see \cite{Boucher-complementary:2020}.
See also \cite{Gruetzner:hal-03963498} and \cite{Boyer-Picaud:2023} for similar considerations.\footnote{Let us emphasize that in the existing work relating to the Knapp--Stein operators as above (alternatively, one may see this as working on the space of Bessel or Riesz potentials), positivity in some form has to be \emph{assumed}, and Property (T) prevents that. The results in this paper do not need such an assumption; we work in the `dual' picture which involves Sobolev spaces.}

\begin{thm}\label{thm:potentials-in-intro}
	Let $(Z,d,\nu)$ be an Ahlfors $D$-regular metric measure space where $\nu$ is the Hausdorff $D$-measure on $(Z,d)$. Let $s\in(0,1)$ with $2s< D$. Then $L^2(Z)$ is dense in $H^{-s}(Z)$, and there exists positive definite integral operator with kernel $k$ on $L^2(Z)$ satisfying $k\asymp d^{-D+2s}$ such that  the norm on $H^{-s}(Z)$ is given by:
	\begin{equation*}
	    \left\|\phi | H^{-s}(Z)\right\|^2=\iint_{Z\times Z}\! \phi(\xi)\overline{\phi(\eta)}k(\xi,\eta)\,
	    \mathrm{d}\nu(\xi)\mathrm{d}\nu(\eta),\quad \text{for all }\phi\in L^2(Z).
	\end{equation*}
	The representation $\pi_{-s}^{(2)}$ of $\Mob(Z)$ defined by \eqref{eq:pi-p-formula-from-intro} extends to a uniformly bounded representation on $H^{-s}(Z)$.
\end{thm}

\subsection{Organisation}

In Section \ref{sect:outline} we give an overview of our approach.
Section \ref{sect:preliminaries} contains the prerequisites. In Section \ref{sect:sobolev-spaces}, we define the fractional Sobolev spaces and show their basic properties. Section \ref{sect:sobolev-ineq} then concerns the fractional Sobolev inequality. In Section \ref{sect:geometric-control} we prove Lorentz norm bounds of certain multipliers, needed in Section \ref{sect:Cayley} where we prove Theorem \ref{thm:Cayley-in-intro}, namely that the operators associated to the Cayley transform are isomorphisms of Sobolev spaces. In Section \ref{sect:representations} we apply this to show that the boundary representations are uniformly bounded, i.e.~Theorem \ref{thm:UB-reps-in-intro}. In Section \ref{sect:almost-invariant-vectors} we prove Theorem \ref{thm:almost-inv-subspace-Lp-in-intro}: for sufficiently large $p$, $\pi_s^{(p)}$ approach the trivial representation. In Section \ref{sect:potential-theory} we use potential theory, when $p=2$, to describe the dual spaces and dual representations in terms of kernel operators analogous to the Knapp--Stein operator. In Appendix \ref{appx:calculations} we give calculations postponed from Section \ref{sect:outline}. In Appendix \ref{appx:groupoids} we clarify the groupoid picture, and finally in Appendix \ref{appx:tangent-and-rescaling} we indicate more formally how can one think of the metric conformal bundle as a tangent space, and explain rescaling of operators.

\subsection{Acknowledgements}
The authors were supported (JS partially) by EPSRC Standard Grant EP/V002899/1. JS would like to thank Bogdan Nica for introducting JS to this subject, and for helpful conversations.

For the purpose of open access, the author has applied a CC BY public copyright licence to any Author Accepted Manuscript version arising from this submission.

\section{Outline of the construction}
\label{sect:outline}

Let $(Z,d,\nu)$ be a bounded Ahlfors $D$-regular metric measure space, where $\nu$ is the Hausdorff $D$-measure.
We sketch our approach to (boundary) representations in the case when $p=2$, $t=0$, and assuming the conformal inversion $d_a$ from \eqref{eq:d_a-defn-in-intro} is a genuine metric on $Z_a=Z\setminus\{a\}$.
We start by explaining the idea and framework around Theorem \ref{thm:UB-reps-in-intro}, which eventually exposes the role of Theorem \ref{thm:Cayley-in-intro} in the story.

For readers familiar with the language of groupoids, the principal idea is the following. Any uniformly bounded representation of an \emph{amenable} group can be unitarised, i.e.~written as a unitary representation conjugated by an invertible. More generally, suppose that a group $G$ acts amenably on $(Z,\nu)$.
Starting with a uniformly bounded representation $\tau$ of $G$ on a Hilbert space $H$, we extend it to a representation of the transformation groupoid $G\ltimes Z$ (on the constant Hilbert bundle over $Z$ with fibre $H$). Applying amenability of $G\ltimes Z$, it splits into a unitary representation of the groupoid $G\ltimes Z$, and an equivariant cocycle over $Z$ which encodes the analytic properties of $\tau$. 

We proceed ``backwards'': by describing a convenient (metric) bundle over $Z$, and showing how a choice of a ``conformal operator on the bundle'' yields an isometric (unitary) representation of $\Mob(Z)\ltimes Z$ on an appropriate Hilbert bundle. Finally, we show how to use the Cayley transform to obtain a uniformly bounded cocycle, yielding uniformly bounded representations of $\Mob(Z)$.

\medskip

The (metric) bundle we will consider is the collection $(Z_a,d_a)_{a\in Z}$ (defined by \eqref{eq:d_a-defn-in-intro}; see also Preliminaries below). We can consider each $(Z_a,d_a)$ as a cotangent space to $Z$ at $a$, see~Appendix \ref{appx:tangent-and-rescaling}. Furthermore, $Z_a$ is an Ahlfors $D$-regular metric measure space when endowed with the measure $\nu_a$ given by $\mathrm{d}\nu_a(\xi)=d^{-2D}(a,\xi)\mathrm{d}\nu(\xi)$.
The bundle carries a natural (diagonal) action of $\Mob(Z)$ by
\begin{equation*}
g(a,\xi)=(ga,g\xi)\in Z_{ga},\quad \text{ for all }g\in\Mob(Z)\text{ and }\xi\in Z_a,
\end{equation*}
which induces an action of the groupoid $\Mob(Z)\ltimes Z$.
For the purposes of this paper, it is sufficient to consider $(Z_a,d_a,\nu_a)_{a\in Z}$ as simply a collection of metric measure spaces. However, let us point out that we can understand it as a genuine bundle by considering the total space
\begin{equation*}
	TZ =_{\text{def}} Z\times Z\setminus \text{diagonal} \simeq (Z_a,d_a)_{a\in Z}.
\end{equation*}
This space carries a $\Mob(Z)$-invariant measure $m$ given as
$m=\int \nu_a \mathrm{d}\nu(a)$ (or $\mathrm{d}m(\xi,\eta) = d^{-2D}(\xi,\eta)\,\mathrm{d}\nu(\xi)\,\mathrm{d}\nu(\eta)$), 
which is uniquely determined when $Z$ is obtained from a geometric action of a hyperbolic group, and called Bowen--Margulis--Sullivan measure \cite{Margulis:1969}, see also \cite[Subsection 7.1]{Nica:2013}.

\begin{rem}
The Koopman (or quasi-regular) representation from this diagonal action of $\Mob(Z)$ on $(TZ,m)$ can be interpreted as a unitary representation of $\Mob(Z)\ltimes Z$ on the Hilbert bundle $(L^2(Z_a,\nu_a))_{a\in Z}$ given by
$$
\phi\mapsto |g'|^{-D}(a)g^*\phi\in L^2(Z_{ga},\nu_{ga}),\quad \text{ for $g\in\Mob(Z)$ and $\phi\in L^2(Z_a,\nu_a)$.}
$$
\end{rem}

For $s\geq0$, we define a \textbf{$(\Mob(Z),s)$-conformal operator} on the bundle $TZ$ to be a collection $(\mathcal{D}_a, T_a)_{a\in Z}$ of densely defined closed operators $T_a$ on $L^2(Z_a,\nu_a)$, such that for all $g\in\Mob(Z)$ and $\phi\in\mathcal{D}_a$,
$g^*\phi\in\mathcal{D}_{ga}$ \footnote{Recall that the notation $g^{*}$ denotes the induced action of $\Mob(Z)$ on functions: for $g\in\Mob(Z)$, $a\in Z$, and $\phi:Z_a\to\mathbb{C}$, we have $g^{*}\phi:Z_{ga}\to\mathbb{C}$ is 	given by $(g^*\phi)(\xi)=\phi(g^{-1}\xi)$.} and
\begin{equation*}
	T_{ga}(g^*\phi) = \D{g}^{2s}(a)\,g^*(T_a\phi).
\end{equation*}
We can think of $(\mathcal{D}_a, T_a)_{a\in Z}$ as a decomposition of a single operator on $L^2(TZ,m)$.

\begin{rem}
This definition is motivated by the following notion of metric conformality: a densely defined operator $T$ on $L^2(Z)$, with a $\Mob(Z)$-invariant domain $\mathcal{D}(T)$, is called \emph{$(\Mob(Z),s,s')$-conformal} (for $s,s'\in\mathbb{R}$); or \emph{$s$-conformal} if $s'=1-s$; if it satisfies (with notation of \eqref{eq:pi_z})
$$
T\pi_s(g)=\pi_{s'}(g)T\quad\text{for }g\in \Mob(Z)\text{ and }\phi\in\mathcal{D}(T).
$$
A fundamental example of an $s$-conformal operator is the Knapp--Stein operator \eqref{eq:Knapp-Stein-in-intro}.
The motivation for considering operators on bundles is a rescaling phenomenon observed for conformal operators on $Z$, see Appendix \ref{appx:tangent-and-rescaling}.
\end{rem}

\begin{exam}\label{expl:fractional-Laplace}
	Suppose that $X$ is a $k$-regular tree for some integer $k\geq 3$, and fix a vertex $o\in X$. The boundary $Z=\partial X$ is the set of infinite geodesic paths starting from $o$. For two such paths, $\xi,\eta\in Z$, denote $\langle\xi,\eta\rangle_o$ the length of the longest common subpath of $\xi$ and $\eta$ (also known as the Gromov product of $\xi$ and $\eta$). Then $d(\xi,\eta)=\exp(-\langle \xi,\eta\rangle)$ is a metric on $Z$ of Hausdorff dimension $D=\ln(k-1)$. Let $\nu$ be the $D$-Hausdorff measure on $(Z,d)$. For $s\geq0$ and $a\in Z$, consider 
	\begin{equation}\label{eq:fract-Laplace-formula}
		\Delta_a^s\phi(\xi) = -\int_{Z_a} \frac{\phi(\eta)-\phi(\xi)}{d_a^{D+2s}(\xi,\eta)}\,\mathrm{d}\nu_a(\eta),
	\end{equation}
	where $\phi: Z_a\to\mathbb{C}$ is locally constant, zero near $a$, and $\xi\in Z_a$. Then $(\Delta_a^s)_{a\in Z}$ (with appropriate domains) is a $(\Mob(Z),s)$-conformal operator; see Calculation \ref{calc:conformal-Delta}. 
\end{exam}

Given an $s$-conformal operator $(\mathcal{D}_a,T_a)_{a\in Z}$ as above, suppose further that each $T_a$ is a positive operator. Then the expression $\left(T_a\phi,\phi\right)_{L^2(Z_a)}$ defines a Hilbert space norm on $\mathcal{D}_a$; denote by $\mathcal{H}_{T_a}$ its completion.
For $g\in\Mob(Z)$ and $a\in Z$, denote
\begin{equation}\label{eq:Pi_s_g_a-defn}
	\Pi_{s}(g;a): \mathcal{H}_{T_a} \to \mathcal{H}_{T_{ga}},\quad \Pi_{s}(g;a)\phi = \D{g}^{\frac{D}{2}-s}(a)\, g^*\phi.
\end{equation}
Then each $\Pi_s(g;a)$ is well defined and unitary (Calculation \ref{calc:pi_z_g_a-is-unitary}), and consequently $\Pi_s$ is a unitary representation of $\Mob(Z)\ltimes Z$, see \eqref{eq:pi_z_g_a-is-a-groupoid-repn} and Appendix \ref{appx:groupoids}. In other words, there is a correspondence between unitary structures for $\Pi_s$ and positive definite $s$-conformal operators over $TZ$.

\medskip

If we further assume that each $T_a$ is an infinitesimal generator of a Dirichlet form,  we can describe the Hilbert spaces $\mathcal{H}_{T_a}$ in terms of these forms. In particular, the fractional Laplacians $(\Delta^s_a)$ from Example \ref{expl:fractional-Laplace} are associated to the ``energy functionals''
\begin{equation*}
	\mathcal{E}_a(\psi)=[\psi]_{s,2}^2=\iint\frac{|\psi(\xi)-\psi(\eta)|^2}{d_a^{D+2s}(\xi,\eta)}\mathrm{d}\nu_a(\xi) \mathrm{d}\nu_a(\eta)\quad\text{for $\psi\in\mathcal{D}_a$},
\end{equation*}
leading to the homogeneous fractional Sobolev spaces $\dt{W}^{s,2}(Z_a)$ from the Introduction. In the remainder of the paper, we consider the forms $[\cdot]_{s,p}$ for $s\in(0,1)$ on general Ahlfors $D$-regular metric measure spaces, and keep the notation $\Delta_a^s$ for the infinitesimal generator of $[\cdot]_{s,2}$. 

\medskip

To obtain a representation of $\Mob(Z)$ from $\Pi_s$, we additionally need an equivariant cocycle, that is, 
a linear map $c_s(b,a): \dt{W}^{s,2}(Z_a) \to \dt{W}^{s,2}(Z_b)$ for all $a,b\in Z$, satisfying the relations $c_s(b,a)c_s(a,d)=c_s(b,d)$, $c_s(a,a)=\mathrm{Id}$, and
\begin{equation}\label{eq:ub-cocycle-equivariant}
	c_s(gb,ga)\Pi_s(g;a) = \Pi_s(g;b)c_s(b,a)
\end{equation}
for all $a,b,d\in Z$, $g\in\Mob(Z)$. Given such a cocycle $c_s$ and $a\in Z$, the formula
\begin{equation}\label{eq:pi_a_s-defn}
	\pi_{a,s}(g) = c_s(a,ga)\,\Pi_s(g;a)
\end{equation}
defines a representation of $\Mob(Z)$ on $\dt{W}^{s,2}(Z_a)$, which is uniformly bounded if $c_s$ is uniformly bounded, i.e.~if $\|c_s(\cdot,\cdot)\|$ is a bounded function on $Z\times Z$.

\begin{rem}
It follows from equivariance that $c(b,a)^{-1}\pi_{b,s}c(b,a)=\pi_{a,s}$
for all $a\in Z$, and thus any choice of $a\in Z$ leads to a similar representation.
\end{rem}

In our case, we obtain the cocycle from the metric derivative of the identification of $Z_a$ and $Z_b$ through the Cayley transforms; namely we let 
\begin{equation}\label{eq:ub-cocycle-with-Cayley}
	c_s(b,a) = \Omega_{s,2}(b)\circ \Omega_{s,2}^{-1}(a)\quad\text{for }a,b\in Z.
\end{equation}
By Theorem \ref{thm:Cayley-in-intro}, $c_s$ is a uniformly bounded cocycle.
$$
\begin{tikzcd}
	\dt{W}^{s,2}(Z_{a})
	\arrow[leftarrow]{rr}{c_{s}(a,ga)}
	\arrow[bend right=-20]{rr}{\Pi_s(g;a)}
	&&\dt{W}^{s,2}(Z_{ga})
	\arrow{dl}{\Omega_{s,2}^{-1}(ga)}
	\\
	& W^{s,2}(Z)
	\arrow{ul}{\Omega_{s,2}(a)}
\end{tikzcd}
$$
The covariance relation \eqref{eq:ub-cocycle-equivariant} follows from a direct calculation (see \ref{calc:pi_z-via-Omega-is-as-wanted}) showing that for $g\in\Mob(Z)$ we actually have
\begin{equation}\label{eq:pi_s-via-Omega}
	\pi_s(g) = \Omega_{s,2}^{-1}(ga)\Pi_s(g;a)\,\Omega_{s,2}(a)
\end{equation}
(as operators on $W^{s,2}(Z)$), where $\pi_s$ is the original boundary representation, given by $\pi_s(g)\phi = \Db{g^{-1}}^{\frac{D}{2}-s}\cdot g^*\phi$. Consequently, it is uniformly bounded on $W^{s,2}(Z)$.

\subsection{A note about the general case}

The $L^p$-version ($p\geq 2$) of the above construction is obtained by considering non-linear $(p-1)$-homogeneous operators, namely the $p$-fractional Laplacian, associated to the energy functional $[\cdot]_{s,p}$ from \eqref{eq:defn-of-[]sp-intro}.
As our approach deals primarily with the energy functionals, the arguments remain the same as in the Hilbert space case.

\section{Preliminaries}\label{sect:preliminaries}

\subsection{Notation}
\label{subsect:prelim:notation}

For non-negative real-values functions $f,g$, we write $f\prec g$ if $\exists c>0$ with $f\leq cg$. We also write $f\asymp g$ when $f\prec g$ and $g\prec f$. In practice the constant(s) $c$ may depend on some of the parameters appearing in the formulas (e.g. $p$, $s$, $D$); this will be clear from the context, or specified as in $\asymp_{s,p}$.

If $(Y,d)$ is a metric space, we denote the balls by $B(x,r)=\{y\in Y: d(x,y)<r\}$. 
We denote by $C(Y)$ the space of continuous functions $\phi:Y\to\mathbb{C}$, by $C_0(Y)\subseteq C(Y)$ the subspace of functions vanishing at infinity, and $C_c(Y)$ the subspace of the compactly supported ones. We consider the latter two endowed with the supremum norm.
We denote by $\Lip(Y)\subseteq C(Y)$ the set of Lipschitz functions, i.e.~functions $\phi:Y\to \mathbb{C}$ for which there exists $L\geq0$ with $|\phi(x)-\phi(y)|\leq L d(x,y)$ for all $x,y\in Y$.

\subsection{Metric measure spaces and M\"obius maps}
\label{subsect:prelim:mms-and-mobius}

Let $(Z,d,\nu)$ be a metric measure space, that is, a complete metric space $(Z,d)$ without isolated points, and a Borel measure $\nu$ on $Z$. 
We will exclusively work with ones which are also Ahlfors $D$-regular, i.e.~for which we have $\nu(B(x,r))\asymp r^D$ for $0\leq r<\diam(Z)$ and $x\in Z$. It follows from \cite[Section 1.4.3]{Mackay-Tyson:2010} that $D$ is then the Hausdorff dimension of $(Z,d)$ and $\nu$ is equivalent to the $D$-dimensional Hausdorff measure on $(Z,d)$. In the context of representations, we need a tighter relation between $d$ and $\nu$, namely that $\nu$ \emph{is} the $D$-dimensional Hausdorff measure, as explained below.

Recall that a M\"obius map is a function $g:(Z,d)\to (Y,d')$ between metric space which preserves cross-ratios, i.e.~the quantity $[x,y;z,w] = \frac{d(x,z)d(y,w)}{d(x,w)d(y,z)}$, whenever this expression makes sense. We denote by $\Mob(Z)$ the group of self-homeomorphisms of $(Z,d)$. For any $g\in \Mob(Z)$ there exists a \emph{metric derivative} of $g$, that is, a Lipschitz function $\D{g}:Z\to(0,\infty)$ satisfying the \emph{geometric mean value property}
\begin{equation}\label{eq:GMV}
	d^2(g\xi,g\eta)=\D{g}(\xi)\,\D{g}(\eta)\, d^2(\xi,\eta)\qquad \text{for all }g\in\Mob(Z),\, \xi,\eta\in Z.
\end{equation}
see e.g.~\cite[Lemmas 6, 7]{Nica:2013} (cf.~also \cite{Sullivan:1979}). Now, if $\nu$ is the $D$-dimensional Hausdorff measure which is also Ahlfors regular, $g$ preserves the class of $\nu$  and (a power of) $\D{g}$ gives the Radon--Nikodym derivative of the action (see e.g.~\cite[Lemma 8]{Nica:2013}):
\begin{equation}\label{eq:change-of-var}
	\mathrm{d}\nu(g\xi) = \D{g}^D(\xi)\,\mathrm{d}\nu(\xi)
	\qquad\text{for all }g\in\Mob(Z)\text{ and }\xi\in Z.
\end{equation}
We further have a cocycle formula
\begin{equation}\label{eq:metric-deriv-cocycle}
	\Db{gh}(\xi) = \D{g}(h\xi)\, \D{h}(\xi)\qquad\text{for all }g,h\in\Mob(Z)\text{ and }\xi\in Z.
\end{equation}
It follows that the formula \eqref{eq:pi-p-formula-from-intro} in Theorem \ref{thm:UB-reps-in-intro} defines a group homomorphism, and (up to adjusting the labelling of the parameters) agrees with \eqref{eq:pi_z}.
Finally, the above formulas imply that we have the following formal adjoint relation (with respect to the $L^2$-pairing) for $g\in \Mob(Z)$:
\begin{align}
	\pi_{z}^{(p)}(g)^* = \pi_{-\overline{z}}^{(q)}(g^{-1}).
\end{align}

As defined in the Introduction, for $a\in Z$ we consider the \emph{Cayley transform} at $a$: denote
\begin{equation}\label{eq:Z_a-d_a-nu_a-defn}
	Z_a=Z\setminus \{a\},
	\quad d_a(\xi,\eta)=\frac{d(\xi,\eta)}{d(\xi,a)d(\eta,a)},
	\quad 	\mathrm{d}\nu_a(\xi)=d^{-2D}(a,\xi)\mathrm{d}\nu(\xi).
\end{equation}
 Then $d_a$ is a (complete) quasi-metric on $Z_a$, i.e.~there exists a metric $\rho_a$ on $Z_a$ satisfying $\frac{1}{4}d_a(\xi,\eta) \leq \rho_a(\xi,\eta) \leq d_a(\xi,\eta)$ for all $\xi,\eta\in Z_a$ by \cite[Lemma 4.5]{Haissinsky:2015}.
However, as we need exact formulas (below), rather than passing to $\rho_a$, we continue to work with quasi-metric (measure) spaces. In particular, the notation $B(x,r)$ for balls in $Z_a$ will denote the ``quasi-balls'' with respect to $d_a$.
(As pointed out in the Introduction, when $Z$ is the Gromov boundary of a strongly hyperbolic, quasi-geodesic, proper metric space \cite{Nica-Spakula:2016}, then $d_a$ is complete metric space by \cite[Theorem 16]{Mineyev:2007}.)

Denote by $\mathcal{C}_a:(Z_a,d_a)\to (Z,d)$ the inclusion. It is a cross-ratio preserving map \cite[Proposition 18]{Mineyev:2007}, and its metric derivative is
\begin{equation}\label{eq:C_a-deriv-formula}
	\D{\mathcal{C}_a}(\xi) = d^{-2}(\xi,a).
\end{equation}
Turning to measures, the formula \eqref{eq:Z_a-d_a-nu_a-defn} for $\nu_a$ 
defines a Borel measure on $Z_a$, which is Ahlfors $D$-regular, i.e.~$\nu_a(B(x,r))\asymp r^D$ for $x\in Z_a$, $0\leq r$, and the constants in `$\asymp$' are uniform over $a\in Z$.

To be able to refer to spaces like these, let us state that by an \emph{Ahlfors $D$-regular quasi-metric measure space} we shall mean a triple $(Y,d,\nu)$, where $d$ is a complete quasi-metric on $Y$, $\nu$ is a Borel measure on $(Y,d)$, and $\nu(B(y,r))\asymp r^D$ for $y\in Y$, $0\leq r<\diam(Y)$. Note that the topology on $Y$ depends only on the equivalence class of $d$, and likewise Ahlfors regularity remains the same (possibly adjusting the $\asymp$-constants) under replacing either/both $d$ and $\nu$ by equivalent quasi-metrics or measures.

We close the subsection with calculation of the transformation rules for the action of $\Mob(Z)$ on the (metric conformal bundle) $(Z_a,d_a,\nu_a)_{a\in Z}$. For $g\in\Mob(Z)$, $a\in Z$, and all $\xi,\eta\in Z_a$, we have
\begin{equation}\label{eq:GMV-for-d_a}
	\begin{aligned}
		d_{ga}(g\xi,g\eta)
		&\stackrel{\eqref{eq:Z_a-d_a-nu_a-defn}}{=} \frac{d(g\xi,g\eta)}{d(g\xi,ga)d(g\eta,ga)}\\
		&\stackrel{\eqref{eq:GMV}}{=} \frac{\D{g}^{\frac{1}{2}}(\xi)\,\D{g}^{\frac{1}{2}}(\eta)}{\D{g}^{\frac{1}{2}}(\xi)\,\D{g}^{\frac{1}{2}}(\eta)\,\D{g}(a)}\cdot\frac{d(\xi,\eta)}{d(\xi,a)d(\eta,a)}
		= \frac{d_a(\xi,\eta)}{\D{g}(a)},
	\end{aligned}
\end{equation}
and similarly
\begin{equation}\label{eq:change-of-var-nu_a}
	\begin{aligned}
		\mathrm{d}\nu_{ga}(g\xi)
		\stackrel{\eqref{eq:Z_a-d_a-nu_a-defn}}{=} \frac{\mathrm{d}\nu(g\xi)}{d^{2D}(ga,g\xi)}
		\stackrel{\eqref{eq:change-of-var}\eqref{eq:GMV}}{=}
		\frac{\D{g}^D(\xi)\mathrm{d}\nu(\xi)}{\D{g}^D(\xi)\,\D{g}^D(a)\, d^{2D}(a,\xi)}
		\stackrel{\eqref{eq:Z_a-d_a-nu_a-defn}}{=}
		\frac{\mathrm{d}\nu_a(\xi)}{\D{g}^{D}(a)}.
	\end{aligned}
\end{equation}

\subsection{Lorentz spaces}
\label{subsect:prelim:lorentz}

Recall that Lorentz spaces \cite{Lorentz:1950,Hunt:1966} are function spaces that interpolate between $L^p$-spaces. We recall the results we will need; for general theory see e.g.~textbook \cite{Grafakos:2014}. Let $(Z,\nu)$ be a non-atomic $\sigma$-finite measure space. For a measurable $f:Z\to\mathbb{C}$, denote its non-increasing rearrangement by $f^*:\mathbb{R}^+\to\mathbb{R}^+$, $f^*(t) = \inf\{s>0\mid \nu(\{|f|>s\})\leq t\}$. For $1\leq p, q\leq\infty$, define the \emph{$(p,q)$--Lorentz norm} by
\begin{equation*}
	\left\|f\,|L(Z;p,q)\right\| = \left\|t\mapsto t^{1/p}f^*(t) \,| L^{q}(\tfrac{\mathrm{d}t}{t})\right\|.
\end{equation*}
We write also $L(p,q)$ when the measure space is clear from the context. One has that $L(p,p)=L^p$, and $L(p,\infty)$ is the \emph{weak $L^p$-space}.

\begin{prop}[{\cite[Proposition 1.4.9]{Grafakos:2014}}]\label{prop:Lorentz-norm-reformulation}
If $1\leq q\leq p<\infty$, then
\begin{equation*}
	\left\|f\,|L(p,q)\right\| = p^{1/q} \left\|s\mapsto s\cdot\nu(\{|f|>s\})^{1/p} \,|L^{q}(\tfrac{\mathrm{d}s}{s})\right\|.
\end{equation*}
\end{prop}

\begin{prop}[Inclusions {\cite[Proposition 1.4.10]{Grafakos:2014}}]
For $1\leq p\leq \infty$ and $1\leq q<r\leq \infty$, one has $L(p,q)\subseteq L(p,r)$.
\end{prop}

\begin{prop}[H\"older-type inequality {\cite[Theorem 1.4.16]{Grafakos:2014}}]\label{prop:Lorentz-Holder-ineq}
	Let $p,p',q,q'\in[1,\infty]$ satisfy $\frac{1}{p}+\frac{1}{p'}=1=\frac{1}{q}+\frac{1}{q'}$. Then for $f\in L(p,q)$ and $h\in L(p',q')$, we have
	$$
	\int |f(\xi)h(\xi)| d\xi \leq \|f \mid L(p,q)\|\cdot \|h \mid L(p',q') \|.
	$$
\end{prop}

\begin{lem}[{\cite[Remark 1.4.7]{Grafakos:2014}}]\label{lem:Lorentz-powers}
	Let $p\in[1,\infty)$, $a\geq1$.
		For any $u\in L(pa,a)$:
		$$
		\left\| |u|^a \mid L(p,1)\right\| = \left\|u\mid L(ap,a)\right\|^a.
		$$
\end{lem}

\section{Fractional Sobolev spaces}\label{sect:sobolev-spaces}

For this section, let $(Y,d,\nu)$ be an Ahlfors $D$-regular quasi-metric measure space.
Fix also $1\leq p<\infty$ and $s\in(0,\infty)$.

Given a measurable function $\phi: Y\to\mathbb{C}$, recall from \eqref{eq:defn-of-[]sp-intro} that its \emph{Gagliardo seminorm} is given by:
\begin{equation}\label{eq:defn-Gagliardo-seminorm-FSSsect}
	[\phi]_{s,p} = \left(\iint_{Y\times Y}\frac{|\phi(\xi)-\phi(\eta)|^p}{d^{D+sp}(\xi,\eta)} \mathrm{d}\nu(\xi)\mathrm{d}\nu(\eta)\right)^{1/p}.
\end{equation}
Recall from the Introduction that the \emph{fractional Sobolev space} $W^{s,p}(Y)$ on $(Y,d,\nu)$ is the function space
$W^{s,p}(Y) = \{\phi \in L^p(Y,\nu) : [\phi]_{s,p} < \infty\}$,
endowed with norm
$\|\phi\,|\,W^{s,p}(Y)\|^p=\|\phi\,|\,L^p\|^p+[\phi]_{s,p}^p$,
also called the Gagliardo, or Aronszajn, or Slobodeckij norm \cite[Section 2]{Sobolev-Hitchhikers}.
We also need the homogeneous fractional Sobolev space $\dt{W}^{s,p}(Y)$, which is defined as the completion of the set of compactly supported, measurable functions $\phi:Z\to\mathbb{C}$ with $[\phi]_{s,p}<\infty$
in the semi-norm $[\cdot]_{s,p}$. Note that for instance constant functions have zero Gagliardo seminorm. For a discussion about homogeneous Sobolev spaces in the context of $\mathbb{R}^n$,  see \cite{Brasco-Gomez-Vazquez:2021} and references therein.

\begin{rem}
	When $Y=\mathbb{R}^n$ ($n\ge1$) with the Euclidean metric and Lebesgue measure, $W^{s,p}(\mathbb{R}^n) = \{0\}$ for $s\geq 1$, see the discussion after \cite[Proposition 2.2]{Sobolev-Hitchhikers}. In this case, for $s\in(0,1)$, we obtain the fractional Sobolev spaces that interpolate between $L^p(\mathbb{R}^n)$ and the Sobolev spaces $W^{1,p}(\mathbb{R}^n)$ of $L^p$-functions with $p$-integrable gradients \cite{Sobolev-Hitchhikers}.
	
	On the other hand if $Y$ is totally disconnected, for instance a finite dimensional vector space over the $q$-adic numbers ($q$ prime) viewed as a metric space
 and its corresponding Hausdorff measure, then $W^{s,p}(Y)$ is dense in $L^p(Y)$ for all $s>0$ \cite{Bendikov-Grigoryan-Hu-Hu:2021}.
\end{rem}

\begin{rem}\label{rem:Wsp-contains-Lip-and-regularity}
	Note that both $W^{s,p}(Y)$ and $\dt{W}^{s,p}(Y)$ both contain compactly supported Lipschitz functions.
	This is a classical consequence of $D$-Ahlfors regularity: The value of $[\phi]_{s,p}$ for any compactly supported Lipschitz function $\phi:Y\to\mathbb{C}$ is bounded from above by an integral for the form $\iint_{A\times A}C/d^{\alpha}(\xi,\eta)\,\mathrm{d}\nu(\xi)\mathrm{d}\nu(\eta)$ for some $0\leq \alpha<D$, a bounded set $A\subset Y$ and a constant $0\leq C<\infty$. Any such integral is finite, cf.~for example Lemma \ref{lem:int-over-ball-calculation}.
	
	Consequently, $W^{s,p}(Y)$ is dense in $L^p(Y)$, and both $C_c(Y)\cap W^{s,p}(Y)$ and $C_c(Y)\cap \dt{W}^{s,p}(Y)$ are dense in $C_0(Y)$. In fact, they are also dense in $W^{s,p}(Y)$ and $\dt{W}^{s,p}(Y)$ respectively, i.e. our Sobolev spaces are \emph{regular}. This has been communicated to us by D.~Gerontogiannis (private communication; a trace theory argument).
\end{rem}

\begin{prop}\label{prop:Wsp-is-a-Banach-space}
	Let $(Y,d,\nu)$ be an Ahlfors $D$-regular quasi-metric measure space, $s\in(0,\infty)$, and $p\in[1,\infty)$. Then
	$W^{s,p}(Y)$, and $\dt{W}^{s,p}(Y)$ when $Y$ is unbounded, are complete. Consequently, they are Banach spaces.
\end{prop}

\begin{proof}
	Let $(u_n)_n$ be a Cauchy sequence in $W^{s,p}(Y)$. Then $(u_n)_n$ is Cauchy in $L^{p}$, and thus converges to some $u\in L^p$ in $L^p$-norm.
	Consequently, one can extract a subsequence, $(v_n)_n$, which converges pointwise to $u$.
	Thus by Fatou's Lemma, we have for any $n\in\mathbb{N}$
	\begin{align*}
		[v_n-u]^p_{s,p}&=\iint\frac{|(v_n(\xi)-u(\xi))-(v_n(\eta)-u(\eta))|^p}{d^{D+sp}(\xi,\eta)} \mathrm{d}\xi \mathrm{d}\eta\\
		&=\iint\liminf_m\frac{|(v_n(\xi)-v_m(\xi))-(v_n(\eta)-v_m(\eta))|^p}{d^{D+sp}(\xi,\eta)} \mathrm{d}\xi \mathrm{d}\eta\\
		&\le\liminf_m\iint\frac{|(v_n(\xi)-v_m(\xi))-(v_n(\eta)-v_m(\eta))|^p}{d^{D+sp}(\xi,\eta)} \mathrm{d}\xi \mathrm{d}\eta\\
		&=\liminf_m\, [v_n-v_m]_{s,p}^p.
	\end{align*}
	Consequently, using the triangle inequality for the Gagliardo seminorm, we have
	\begin{align*}
		[u_n-u]_{s,p}&\le [u_n-v_n]_{s,p}+[v_n-u]_{s,p}\\
		&\le [u_n-v_n]_{s,p}+\liminf_m[v_n-v_m]_{s,p}\xrightarrow{n\rightarrow \infty} 0.
	\end{align*}
	Thus $u\in W^{s,p}(Y)$ and $u_n\to u$ in $W^{s,p}(Y)$.

	For $\dt{W}^{s,p}(Y)$, the claim follows from an analogous argument, first applying the Fractional Sobolev Inequality (Theorem \ref{thm:Sobolev-inequality}) which provides a continuous embedding of $\dt{W}^{s,p}(Y)$ into $L^{p^*}(Y)$ for $p^*=\frac{pD}{D-sp}$ when $Y$ is unbounded.
\end{proof}

\begin{prop}\label{prop:Wsp-is-a-closed-subsp-of-L^p}
	Let $(Y,d,\nu)$ be an Ahlfors $D$-regular quasi-metric measure space, $s\in(0,\infty)$, and $p\in[1,\infty)$.
	Then both $W^{s,p}(Y)$, and $\dt{W}^{s,p}(Y)$ when $Y$ is unbounded, are closed subspaces of an $L^p$-space. Consequently, they are reflexive.
\end{prop}
\begin{proof}
	Define a measure on $Y\times Y$ by $\mathrm{d}\mu(\xi,\eta) = d^{-D-sp}(\xi,\eta)\mathrm{d}\nu(\xi)\mathrm{d}\nu(\eta)$. By standard arguments, this is a $\sigma$-finite Borel measure on $Y\times Y$. For a measurable $\phi:Y\to\mathbb{C}$, denote $D(\phi)(\xi,\eta) = \phi(\xi)-\phi(\eta)$.
	Then the linear maps
	\begin{align*}
		T&:W^{s,p}(Y)\to L^p(Y)\oplus L^p(Y\times Y; \mu), &\phi\mapsto (\phi,D(\phi))\\
		D&:\dt{W}^{s,p}(Y)\to L^p(Y\times Y; \mu)
			\end{align*}
	are isometries (just from the definitions of the norms), and thus have closed ranges.
\end{proof}

\section{A tight fractional Sobolev inequality}\label{sect:sobolev-ineq}

The purpose of this section is to prove the following Theorem. In this section, the symbols $\prec,\asymp$ mean `up to a multiplicative constant depending only on the space, $s$, and $p$'.

The results in this section are stated for Ahlfors regular metric measure spaces, but the conclusion(s) follow for Ahlfors $D$-regular quasi-metric measure spaces as well, replacing a metric $d$ by an equivalent quasi-metric $\rho$ only requires possibly adjusting the constants hidden in `$\prec$'.

\begin{thm}\label{thm:Sobolev-inequality}
	Let $(Z,d,\nu)$ be an Ahlfors $D$-regular metric measure space (not necessarily bounded), $s\in(0,\infty)$, and $p\geq 1$, such that $sp<D$. Denote $p^*=\frac{pD}{D-sp}$. Then for any compactly supported measurable function $f:Z\to\mathbb{C}$ we have
	\begin{equation}\label{eq:Sobolev-inequality}
		\left\|f | L(p^*,p)\right\|^p \prec \nu(Z)^{\frac{-sp}{D}}\left\|f | L^p\right\|^p + \iint_{Z^2} \frac{\left|f(\xi)-f(\eta)\right|^p}{d^{D+sp}(\xi,\eta)}\, \mathrm{d}\nu(\xi)\mathrm{d}\nu(\eta).
	\end{equation}
\end{thm}

There is an extensive literature about fractional Sobolev inequality in the context of (open subsets of) $\mathbb{R}^n$; often having $\|f|L^{p^*}\|$ on the left-hand side, and without the $L^{p}$ term on the right-hand side (which also disappears in \eqref{eq:Sobolev-inequality} when $Z$ is unbounded). In view of the continuous inclusion $L(p^*,p)\subseteq L(p^*,p^*) =L^{p^*}$ \cite[Proposition 1.4.10]{Grafakos:2014}, the $L^{p^*}$-version follows from the Lorentz $L(p^*,p)$-version. This tighter Lorentz version is also well known in the context of $\mathbb{R}^n$, see e.g.~\cite{Peetre:1966}. 

For a proof, we build on a proof of a fractional Sobolev inequality for $\mathbb{R}^n$ in \cite[Theorem 6.5]{Sobolev-Hitchhikers}. This proof uses very little analysis in $\mathbb{R}^n$, and we can use large parts of it without modification. We opted not to reproduce the complete proof here: the original is very well written and detailed, and we do recommend that the reader consults the original. We outline the main steps, and the parts of the proof that we need to modify to our context.

Let us outline the main structure of the proof from \cite[Chapter 6]{Sobolev-Hitchhikers}.
\begin{enumerate}
	\item One can assume that $f\in L^\infty(Z)$  using a standard argument, using \cite[Lemma 6.4]{Sobolev-Hitchhikers}. This carries over verbatim.
	\item One ``discretizes'': For $k\in\mathbb{Z}$, denote $A_k=\{|f|>2^k\}$ and $a_k = \nu(A_k)$.
	Then by Proposition \ref{prop:Lorentz-norm-reformulation} we have
	\begin{align*}
		\left\|f|L(p^*,p)\right\|^p
		&= p^*\int_0^\infty s^p\,\nu\left(|f|>s\right)^{{p}/{p^*}}\tfrac{\mathrm{d}s}{s}\\
		&\leq p^*\sum_{k\in\mathbb{Z}}\int_{2^{k}}^{2^{k+1}} s^{p-1} a_k^{{p}/{p^*}}\mathrm{d}s
		 = \tfrac{p^*\left(2^p-1\right)}{p}\cdot \sum_{k\in\mathbb{Z}}2^{kp}a_{k}^{p^*/p}.
	\end{align*}
	This is in fact slightly more straightforward than the original proof, as \cite{Sobolev-Hitchhikers} aims for $\|f|L^{p*}\|$ instead (cf.~steps to obtain \cite[(6.24)]{Sobolev-Hitchhikers}).
	\item As in \cite{Sobolev-Hitchhikers} one then applies a Lemma about sequences of numbers, \cite[Lemma 6.2]{Sobolev-Hitchhikers}, to further get
	\begin{equation*}
		\|f|L(p^*,p)\| \leq C \sum_{k\in\mathbb{Z}, a_k\not=0} 2^{kp}a_{k+1}a_k^{-\frac{sp}{n}}.
	\end{equation*}
	The final (delicate) part is showing that the right-hand side above bounds the right-hand side of the desired inequality \eqref{eq:Sobolev-inequality} from below; an analogue of \cite[Lemma 6.3]{Sobolev-Hitchhikers}: this is Lemma \ref{lem:Hitchhikers-Lemma-6.3} below. Modulo this Lemma, the proof of Theorem \ref{thm:Sobolev-inequality} is done. \hfill\qedsymbol
\end{enumerate}

\begin{lem}\label{lem:Hitchhikers-Lemma-6.3}
	Let $(Z,d,\nu)$ be an Ahlfors $D$-regular metric measure space, $f\in L^\infty(Z)$ with compact support, $s\in(0,1)$, $p\geq 1$, $sp<D$. For $k\in\mathbb{Z}$, denote $A_k=\{|f|>2^k\}$ and $a_k = \nu(A_k)$. Then
	\begin{equation*}
		\iint_{Z^2} \frac{\left|f(\xi)-f(\eta)\right|^p}{d^{D+sp}(\xi,\eta)}\mathrm{d}\nu(\xi)\mathrm{d}\nu(\eta)
		+\nu(Z)^{-{sp}/{D}}\left\|f|L^p\right\|^p
		\succ \sum_{\substack{k\in\mathbb{Z}\\a_k\not=0}}2^{kp}a_{k+1}a_{k}^{-sp/D}.
	\end{equation*}
\end{lem}

\begin{proof}
	We slightly modify the proof of \cite[Lemma 6.3]{Sobolev-Hitchhikers}; we keep the notation consistent with that proof, and focus on the parts which need adjustments. The proof starts by discretizing as before, i.e.~for $k\in\mathbb{Z}$ denoting $A_k=\{|f|>2^k\}$ and $a_k = \nu(A_k)$. Further subdivide into $D_k=\{2^{k+1}\geq|f|> 2^{k}\}$ and denote $d_k=\nu(D_k)$. For brevity, denote $K=\nu(Z)^{-\frac{sp}{D}}$. We finish the preparation by observing that
	\begin{equation}\label{localeq:sum-bound-for-Hitchhiker-L6.3}
		\sum_{i\in\mathbb{Z}} d_i2^{pi}
		\leq \sum_{i\in\mathbb{Z}} \int_{D_i} |f(\xi)|^p \mathrm{d}\nu(\xi)
		=\left\|f|L^p\right\|^p.
	\end{equation}
	
	For $(x,y)\in D_i\times D_j$ with $j\leq i-2$ one has $|f(x)-f(y)|\geq 2^{i-1}$, and thus
	\begin{align}
		\sum_{j\in\mathbb{Z}\atop  j\leq i-2}\int_{D_j}\frac{|f(x)-f(y)|^p}{d(x,y)^{D+sp}}\mathrm{d}\nu(y)
		&\geq 2^{p(i-1)}\sum_{j\in\mathbb{Z}\atop j\leq i-2} \int_{D_j}\frac{\mathrm{d}\nu(y)}{d(x,y)^{D+sp}} \notag\\
		&= 2^{p(i-1)}\int_{Z\setminus A_{i-1}}\frac{\mathrm{d}\nu(y)}{d(x,y)^{D+sp}} \notag\\
		&\geq c_0 2^{pi}a_{i-1}^{-sp/n} - c_1K2^{pi} \label{localeq:Hitchhiker-pre6.15}
	\end{align}
	for suitable constants $c_0,c_1>0$ by Lemma \ref{lem:Hitchhikers-Lemma-6.1} below. Now we execute the main scheme of the proof of \cite[Lemma 6.3]{Sobolev-Hitchhikers}: obtain the analogues of (6.15)--(6.17) using our \eqref{localeq:Hitchhiker-pre6.15}.
	For any $i\in\mathbb{Z}$ we thus have
	\begin{equation}\label{localeq:Hitchhiker-6.15}
		\sum_{j\in\mathbb{Z}\atop j\leq i-2}\iint_{D_i\times D_j}\frac{|f(x)-f(y)|^p}{d(x,y)^{D+sp}}\mathrm{d}\nu(x)\mathrm{d}\nu(y)
		\geq c_0 2^{pi}a_{i-1}^{-sp/n}d_i - c_1K2^{pi}d_i.
	\end{equation}
	Denoting $S=\sum_{i\in\mathbb{Z}, a_{i-1}\not=0}2^{pi}a_{i-1}^{-sp/D}d_{i}$ as in \cite[(6.12)]{Sobolev-Hitchhikers}, summing over $i$, and using \eqref{localeq:sum-bound-for-Hitchhiker-L6.3}, we further get
	\begin{equation}\label{localeq:Hitchhiker-6.17}
		\sum_{i\in\mathbb{Z}\atop a_{i-1}\not=0}\sum_{j\in\mathbb{Z}\atop j\leq i-2}\iint_{D_i\times D_j} \frac{|f(x)-f(y)|^p}{d(x,y)^{D+sp}}\mathrm{d}\nu(x)\mathrm{d}\nu(y)
		\geq c_0 S - c_1K \|f|L^p\|^p.  %
	\end{equation}
	Coming back to \eqref{localeq:Hitchhiker-6.15} and noting that
	$d_i=a_i-\sum_{\ell\in\mathbb{Z},\, \ell\geq i+1} d_{\ell}$, we also get
	\begin{multline}\label{localeq:Hitchhiker-6.16}
		\sum_{j\in\mathbb{Z}\atop j\leq i-2}\iint_{D_i\times D_j}\frac{|f(x)-f(y)|^p}{d(x,y)^{D+sp}}\mathrm{d}\nu(x)\mathrm{d}\nu(y)
		\\\geq c_0\left(2^{pi}a_{i-1}^{-sp/D}a_i-\textstyle\sum_{\ell\in\mathbb{Z},\,\ell\geq i+1} 2^{pi} a_{i-1}^{-sp/D}d_{\ell} \right) -c_1K2^{pi}d_i.
	\end{multline}
	Before combining the above, recall that in \cite[(6.14)]{Sobolev-Hitchhikers}, the authors show that $S\geq \sum_{i\in\mathbb{Z},\, a_{i-1}\not=0}\sum_{\ell\in\mathbb{Z},\,\ell\geq i+1} 2^{pi}a_{i-1}^{-sp/D}d_{\ell}$. Using this, \eqref{localeq:sum-bound-for-Hitchhiker-L6.3}, and \eqref{localeq:Hitchhiker-6.17} in what follows, we start by summing over $i$ in \eqref{localeq:Hitchhiker-6.16}:
	\begin{align*}
		\sum_{i\in\mathbb{Z}\atop a_{i-1}\not=0}&\sum_{j\in\mathbb{Z}\atop j\leq i-2}\iint_{D_i\times D_j} \frac{|f(x)-f(y)|^p}{d(x,y)^{D+sp}}\mathrm{d}\nu(x)\mathrm{d}\nu(y)
		\\
		&\geq c_0\left[\sum_{i\in\mathbb{Z}\atop a_{i-1}\not=0}2^{pi}a_{i-1}^{-sp/D}a_i-\sum_{i\in\mathbb{Z}\atop a_{i-1}\not=0}\sum_{\ell\in\mathbb{Z}\atop \ell\geq i+1} 2^{pi} a_{i-1}^{-sp/D}d_{\ell} \right] -c_1K\!\!\!\sum_{i\in\mathbb{Z}\atop a_{i-1}\not=0}\!\!\! 2^{pi}d_i\\
		&\geq c_0\left[\sum_{i\in\mathbb{Z}\atop a_{i-1}\not=0}2^{pi}a_{i-1}^{-sp/D}a_i-S \right] -c_1K\|f|L^p\|^p \\
		&\geq c_0\sum_{i\in\mathbb{Z}\atop a_{i-1}\not=0}2^{pi}a_{i-1}^{-sp/D}a_i
		- \sum_{i\in\mathbb{Z}\atop a_{i-1}\not=0}\sum_{j\in\mathbb{Z}\atop j\leq i-2}\iint_{D_i\times D_j} \frac{|f(x)-f(y)|^p}{d(x,y)^{D+sp}}\mathrm{d}\nu(x)\mathrm{d}\nu(y)\\
		&\hspace{1em} -2c_1K\|f|L^p\|^p.
	\end{align*}
	The rest of the proof proceeds as in \cite[Lemma 6.1]{Sobolev-Hitchhikers}: we bring the negative terms from the right-hand side to the left-hand side, adjust the constants, and use the discretization by $D_i\times D_j$ and symmetry (cf.~\cite[(6.19)]{Sobolev-Hitchhikers}) to obtain the required inequality.
\end{proof}

In the above proof, we have used the following analogue of \cite[Lemma 6.1]{Sobolev-Hitchhikers}.
\begin{lem}\label{lem:Hitchhikers-Lemma-6.1}
	Let $(Z,d,\nu)$ be an Ahlfors $D$-regular metric measure space. Then for any $s\in(0,1)$, $p\geq 1$, and $E\subset Z$ with $\nu(E)<\infty$, and any $\xi\in Z$, we have
	\[
	\int_{Z\setminus E} \frac{\mathrm{d}\nu(\eta)}{d^{D+sp}(\xi,\eta)}
	\geq
	C\nu(E)^{-sp/D}-\nu(Z)^{-sp/D}
	\]
	for a suitable constant $C=C(Z,sp)>0$, which is uniformly bounded for $sp$ in any fixed compact interval in $(0,\infty)$.
\end{lem}
\begin{proof}
	The proof follows the one of \cite[Lemma 6.1]{Sobolev-Hitchhikers}.
	Denote by $c\geq1$ the constant from Ahlfors regularity, i.e.~we have $c^{-1}r^D\leq \nu(B(\xi,r))\leq c r^D$ for all $\xi\in Z$ and $0\leq r\leq\diam(Z)$.
    Hence, with $\alpha=c^{1/D}$, we have $\nu(B(\xi,\ga r))\geq c^{-1}\alpha^D r^D\geq r^D$.
	Denote $B=B(\xi,\ga r)$ with $r=\nu(E)^\frac{1}{D}$. Then $\nu(B) \geq \nu(E)$, and so
	\begin{align*}
		\nu\left((Z\setminus E)\cap B\right)&=\nu(B)-\nu(E\cap B)
		\geq \nu(E)-\nu(E\cap B)
		= \nu\left(E\cap (Z\setminus B)\right).
	\end{align*}
It follows that
\begin{align*}
\int_{Z\setminus E}\frac{\mathrm{d}\nu(\eta)}{ d^{D+sp}(\xi,\eta)}
&=\int_{(Z\setminus E)\cap B}\frac{\mathrm{d}\nu(\eta)}{ d^{D+sp}(\xi,\eta)} + \int_{(Z\setminus E)\cap (Z\setminus B)}\frac{\mathrm{d}\nu(\eta)}{ d^{D+sp}(\xi,\eta)}\\
&\geq (\alpha r)^{-D-sp}\nu\left((Z\setminus E)\cap B\right) + \int_{(Z\setminus E)\cap (Z\setminus B)}\frac{\mathrm{d}\nu(\eta)}{ d^{D+sp}(\xi,\eta)}\\
&\geq (\alpha r)^{-D-sp}\nu\left(E\cap (Z\setminus B)\right) + \int_{(Z\setminus E)\cap (Z\setminus B)}\frac{\mathrm{d}\nu(\eta)}{ d^{D+sp}(\xi,\eta)}\\
&\geq \int_{E\cap (Z\setminus B)}\frac{\mathrm{d}\nu(\eta)}{ d^{D+sp}(\xi,\eta)} + \int_{(Z\setminus E)\cap (Z\setminus B)}\frac{\mathrm{d}\nu(\eta)}{ d^{D+sp}(\xi,\eta)}\\
&=\int_{Z\setminus B}\frac{\mathrm{d}\nu(\eta)}{ d^{D+sp}(\xi,\eta)};
\intertext{with Lemma \ref{lem:Sobolev-for-a-ball} below, we continue:}
&\geq C \nu(B)^{-sp/D} - \nu(Z)^{-sp/D}\\
&\geq C c^{-sp/D}\alpha^{-sp} r^{-sp} - \nu(Z)^{-sp/D}\\
&\geq C' \nu(E)^{-sp/D} - \nu(Z)^{-sp/D}. \qedhere
\end{align*}
\end{proof}

The following Lemma was used in the above proof; however it will be useful elsewhere as well.
\begin{lem}\label{lem:Sobolev-for-a-ball}
	Let $(Z,d,\nu)$ be an Ahlfors $D$-regular quasi-metric measure space. Then for any $0<\alpha$, any $\xi\in Z$ and $0\leq r\leq\diam(Z)$, we have
\begin{equation*}
		\int_{Z\setminus B(\xi,r)}\frac{\mathrm{d}\nu(\eta)}{d^{D+\alpha}(\xi,\eta)}
	+\nu(Z)^{-\alpha/D}
	\asymp_{\alpha} \nu(B(\xi,r))^{{-\alpha}/{D}}.
\end{equation*}
\end{lem}

\begin{proof}
Using the layer cake formula, we calculate
	\begin{align*}
	\int_{Z\setminus B(\xi,r)}\frac{\mathrm{d}\nu(\eta)}{d^{D+\alpha}(\xi,\eta)}
	&=\int_{\diam(Z)^{-D-\alpha}}^{r^{-D-\alpha}} \nu\left(\{d^{-D-\alpha}(\xi,\cdot) >t \} \right)\mathrm{d}t\\
	&=\int_{\diam(Z)^{-D-\alpha}}^{r^{-D-\alpha}} \nu\left(\{d(\xi,\cdot) < t^{-1/(D+\alpha)} \} \right)\mathrm{d}t\\
	&\asymp \int_{\diam(Z)^{-D-\alpha}}^{r^{-D-\alpha}} t^{-D/(D+\alpha)}\mathrm{d}t 
	= \tfrac{D+\alpha}{\alpha}\left[ t^{\frac{\alpha}{D+\alpha}} \right]_{\diam(Z)^{-D-\alpha}}^{r^{-D-\alpha}} \\
	&=\tfrac{D+\alpha}{\alpha}\left(r^{-\alpha}-\diam(Z)^{-\alpha}\right).
	\end{align*}
	The conclusion now follows as $r^{-\alpha} \asymp_\alpha \nu(B(\xi,r))^{-\alpha/D}$ and $\diam(Z)^{-\alpha} \asymp_\alpha \nu(Z)^{-\alpha/D}$.
\end{proof}

\section{Geometric control}\label{sect:geometric-control}

In this section, the aim is to prove Propositions \ref{prop:geometric-control-1} and \ref{prop:geometric-control-2}, i.e.~Lorentz norm bounds of certain multipliers that will be used in the proof of boundedness of the operators induced by the Cayley transform in the next Section.

\begin{lem}\label{lem:int-over-ball-calculation}
	Let $(Z,d,\nu)$ be an Ahlfors $D$-regular quasi-metric measure space. Then for any $0<\alpha <D$, $\xi\in Z$, and $0\leq r\leq \diam(Z)$, we have
	$$\int_{B(\xi,r)}\frac{\mathrm{d}\nu(\eta)}{d^{\alpha}(\xi,\eta)}\asymp \tfrac{\alpha}{D-\alpha}r^{D-\alpha}.$$
\end{lem}
\begin{rem}\label{rem:int-over-ball-calculation}
	When $\alpha\leq 0$, we still have a (trivial) upper bound $$\int_{B(\xi,r)}d^{-\alpha}(\xi,\eta)\mathrm{d}\nu(\eta)\prec r^{D-\alpha}.$$
\end{rem}
\begin{proof}
Using the layer cake formula, we have
\begin{align*}
	\int_{B(\xi,r)}\frac{\mathrm{d}\nu(\eta)}{d^{\alpha}(\xi,\eta)}
	&=\int_{r^{-\alpha}}^\infty \nu\left(\{d^{-\alpha}(\xi,\cdot) >t \} \right)\mathrm{d}t\\
	&=\int_{r^{-\alpha}}^\infty \nu\left(\{d(\xi,\cdot) < t^{-1/\alpha} \} \right)\mathrm{d}t\\
	&\asymp \int_{r^{-\alpha}}^{\infty} t^{-D/\alpha}\mathrm{d}t 
	=\tfrac{\alpha}{\alpha-D}\left[t^{(\alpha-D)/\alpha}\right]_{r^{-\alpha}}^{\infty}
	=\tfrac{\alpha}{D-\alpha}r^{D-\alpha}.  \qedhere
\end{align*}
\end{proof}

\begin{rem}\label{rem:d-sp-in-Lorentz}
	Let us record the classical observation that in an Ahlfors $D$-regular quasi-metric measure space $(Z,d,\nu)$, for $\alpha\geq 0$ we have that the function $\eta \mapsto d^{-\alpha}(\xi,\eta)$ belongs to $L(D/\alpha,\infty)$ for any $\xi\in Z$, with the Lorentz norm bounded uniformly in $\xi\in Z$. This follows from $\nu\left(\left\{\zeta\in Z : d^{-\alpha}(\xi,\zeta) > t\right\}\right) \prec t^{-D/\alpha}$ for $t\geq 0$.
\end{rem}

\begin{prop}\label{prop:geometric-control-1}
	Let $(Z,d,\nu)$ be an Ahlfors $D$-regular metric measure space.
	Let $\sigma+it\in\mathbb{C}$, $s\in(0,1)$, and $1\le p<\infty$, be such that $\sigma\geq0$, $\sigma p<D$, and $sp<D$. Then
	\begin{equation*}
		\int_Z \frac{\left|1-\big[\frac{ d(a,\eta)}{ d(a,\xi)}\big]^{\sigma+it}\right|^p}{ d^{D+sp}(\xi,\eta)}\mathrm{d}\nu(\xi)
	\prec_{\sigma,t,s,p} d^{-sp}(a,\eta),
	\end{equation*}
	for $a,\eta\in Z$. 
	Consequently, by Remark \ref{rem:d-sp-in-Lorentz}, 
	\begin{equation*}
		\left\|\int_Z \frac{\left|1-\big[\frac{ d(a,\cdot)}{d(a,\xi)}\big]^{\sigma+it}\right|^p}{ d^{D+sp}(\xi,\cdot)} \mathrm{d}\nu(\xi)  \left|L\left(Z;\tfrac{D}{sp},\infty\right)\right.\right\|
	\end{equation*}	is uniformly bounded for $a\in Z$.
\end{prop}

\begin{proof}
	Fix $a,\eta\in Z$. Denote
	$$
	U(\eta)= B(\eta, d(a,\eta)/2), \quad U(a) = B(a, d(a,\eta)/2).
	$$
	Consider $\xi\in U(a)$. Then $d(a,\xi)<\frac{1}{2}d(a,\eta) \leq \frac{1}{2}(d(a,\xi)+d(\xi,\eta))$, and hence $d(a,\xi)\leq 2d(\xi,\eta)$.  Next, we have $d(a,\eta) \leq d(a,\xi)+d(\xi,\eta)\leq 3 d(\xi,\eta)$. Conversely, $d(\xi,\eta) \leq d(a,\xi)+d(a,\eta)\leq \frac{3}{2}d(a,\eta)$. Summarising, we have
	\begin{equation}\label{localeq:Ua-updown-bounds}
		\xi\in U(a)\implies d(a,\xi)\prec d(a,\eta)\asymp d(\xi,\eta). %
	\end{equation}
	Symmetrically, we have
	\begin{equation}\label{localeq:Ueta-updown-bounds}
		\xi\in U(\eta)\implies d(\xi,\eta) \prec d(a,\eta)\asymp d(a,\xi). %
	\end{equation}
	Finally, in $Z\setminus (U(a)\cup U(\eta))$, a similar calculation with the triangle inequality gives
	\begin{equation}\label{localeq:ZminusUs-updown-bounds}
		d(a,\eta)\leq 2d(a,\xi)\asymp d(\xi,\eta). %
	\end{equation}
	The triangle inequality\footnote{It is at this point where a genuine metric is required; \eqref{localeq:1-d-over-d-triangle} may not hold for a quasi-metric.} also implies that
	\begin{equation}\label{localeq:1-d-over-d-triangle}
		\left|1-\tfrac{ d(a,\xi)}{ d(a,\eta)}\right| \leq \tfrac{d(\xi,\eta)}{d(a,\eta)}.
	\end{equation}
	We now estimate the integral in the statement of the Proposition on three subsets of $Z$.

	On $W=Z\setminus (U(a) \cup U(\eta))$, we have that $\frac{d(a,\eta)}{d(a,\xi)}\leq 2$, and using Lemma \ref{lem:Sobolev-for-a-ball} (as $sp<D$), we estimate
	\begin{multline*}
		\int_W\frac{\left|1-\big[\frac{ d(a,\eta)}{ d(a,\xi)}\big]^{\sigma+it}\right|^p}{ d^{D+sp}(\xi,\eta)}\mathrm{d}\nu(\xi)\\
		\prec_{\sigma,p}\int_{Z\setminus U(\eta)}\frac{1}{ d^{D+sp}(\xi,\eta)}\mathrm{d}\nu(\xi)
		\prec_{s,p} \nu(U(\eta))^{\frac{-sp}{D}}\asymp\frac{1}{ d^{sp}(a,\eta)}.
	\end{multline*}
	On $U(a)$ using Lemma \ref{lem:int-over-ball-calculation} (as $\sigma p<D$), and \eqref{localeq:Ua-updown-bounds}, we have
	\begin{multline*}
		\int_{U(a)}\frac{\left|1-\big[\frac{ d(a,\eta)}{ d(a,\xi)}\big]^{\sigma+it}\right|^p}{ d^{D+sp}(\xi,\eta)}\mathrm{d}\nu(\xi)\\
		\prec_{s,p} \frac{\nu(U(a))}{d^{D+sp}(a,\eta)}
		+\frac{d^{\sigma p}(a,\eta)}{ d^{D+sp}(a,\eta)} \int_{U(a)}\frac{1}{d^{\sigma p}(a,\xi)} \mathrm{d}\nu(\xi)
		\prec_{\sigma,p}\frac{1}{ d^{sp}(a,\eta)}.
	\end{multline*}
	Finally, we deal with $U(\eta)$. The Mean Value Theorem implies that $|1-x^{\sigma+it}| \prec_{\sigma,t} |1+x|$ for $x\in [\frac{2}{3},3]$. Using this together with \eqref{localeq:Ueta-updown-bounds}, \eqref{localeq:1-d-over-d-triangle}, and Lemma \ref{lem:int-over-ball-calculation}, we have
	\begin{multline*}
		\int_{U(\eta)}\frac{\left|1-\big[\frac{ d(a,\eta)}{ d(a,\xi)}\big]^{\sigma+it}\right|^p}{ d^{D+sp}(\xi,\eta)}\mathrm{d}\nu(\xi)
		\prec_{\sigma,t}
		\int_{U(\eta)}\frac{\left|1-\frac{ d(a,\eta)}{ d(a,\xi)}\right|^p}{ d^{D+sp}(\xi,\eta)}\mathrm{d}\nu(\xi)\\
		\leq
		\int_{U(\eta)} \frac{\mathrm{d}\nu(\xi)}{ d^{p}(a,\xi)d^{D-(1-s)p}(\xi,\eta)}
		\prec_{\sigma,t,s,p}\frac{1}{ d(a,\eta)^{sp}}.
		\qedhere
	\end{multline*}	
\end{proof}

\begin{lem}\label{lem:d^kalpha-in-weak-L-Z_a-D/beta}
	Let $(Z,d,\nu)$ be a bounded Ahlfors $D$-regular metric measure space and consider $(Z_a,d_a,\nu_a)$ as defined in \eqref{eq:Z_a-d_a-nu_a-defn}. Let $0<\alpha\leq D$ and $k\geq 1$. Then
	$$
	d^{k\alpha}(\cdot,a)\in L\Big(Z_a;\tfrac{D}{\alpha}, \infty\Big)
	$$
	with norm bounded uniformly on $a\in Z$.
\end{lem}

\begin{proof}
	We need to check that 
	$t\mapsto t^{D/\alpha}\cdot\nu_a\left(\left\{d^{k\alpha}(\cdot,a)>t\right\}\right)$ is bounded. We can assume $t<\diam(Z)^{k\alpha}$.
	We calculate
	\begin{align*}
		\nu_a\left(\left\{d^{k\alpha}(\cdot,a)>t\right\}\right)
		&= \int_{\{d^{k\alpha}(\cdot,a)>t\}}\hspace{-1.5em} d^{-2D}(\xi,a) \, \mathrm{d}\nu(\xi)\\
		&\leq \int_{0}^{t^{-(2D)/(k\alpha)}}\hspace{-1.5em}  \nu\left(\left\{ d^{-2D}(\cdot,a) > s \right\}\right)\mathrm{d}s\\
		&= \int_{0}^{t^{-(2D)/(k\alpha)}}\hspace{-1.5em}  \nu\left(\left\{ d(\cdot,a) < s^{-1/(2D)} \right\}\right)\mathrm{d}s\\
		&\prec \int_{0}^{t^{-(2D)/(k\alpha)}}\hspace{-1.5em}  s^{-1/2}\mathrm{d}s
		= 2 t^{-\frac{D}{k\alpha}}.
	\end{align*}
	Multiplying the final expression by $t^{D/\alpha}$ yields power function with exponent $\frac{(k-1)D}{k\alpha}\geq0$, which is bounded as $t<\diam(Z)^{k\alpha}<\infty$.
\end{proof}

\begin{prop}\label{prop:geometric-control-2}
	Let $(Z,d,\nu)$ be a bounded Ahlfors $D$-regular metric measure space and consider $(Z_a,d_a,\nu_a)$ as defined in \eqref{eq:Z_a-d_a-nu_a-defn}.
	Let $\sigma+it\in\mathbb{C}$, $s\in(0,1)$, $1\le p<\infty$ be such that $0\leq \sigma$, $\sigma p<D$, and $sp<D$. Then
	$$
	\int_{Z_a} \frac{\left|1-\big[\frac{ d(a,\xi)}{ d(a,\eta)}\big]^{\sigma+it}\right|^p}{ d_a^{D+sp}(\xi,\eta)}\mathrm{d}\nu_a(\xi)
	\prec_{\sigma,t,s,p} d^{sp}(a,\eta),
	$$
	for $a,\eta\in Z$.
	Consequently, by Lemma \ref{lem:d^kalpha-in-weak-L-Z_a-D/beta},
	$$
	\left\|\int_{Z_a} \frac{\left|1-\big[\frac{ d(a,\xi)}{d(a,\cdot)}\big]^{\sigma+it}\right|^p}{ d_a^{D+sp}(\xi,\cdot)} \mathrm{d}\nu_a(\xi)  \left|L\Big(Z_a;\tfrac{D}{sp},\infty\Big)\right.\right\|
	$$
	is uniformly bounded for $a\in Z$.
\end{prop}

\begin{proof}
	As in Proposition \ref{prop:geometric-control-1}, let $U(\eta)=\{\xi\in Z_a\mid d(\xi,\eta)<d(a,\eta)/2\}$ and $U(a)=\{\xi\in Z_a\mid d(\xi,a)<d(a,\eta)/2\}$.
	We proceed again by estimating the integral on three subsets of $Z$. The first step of all three calculations is
	\begin{multline}
		\int\frac{\left|1-\big[\frac{ d(a,\xi)}{ d(a,\eta)}\big]^{\sigma+it}\right|^p}{ d_a^{D+sp}(\xi,\eta)}\,\mathrm{d}\nu_a(\xi)\\
		= \int\frac{\left|1-\big[\frac{ d(a,\xi)}{ d(a,\eta)}\big]^{\sigma+it}\right|^p}{ d^{D+sp}(\xi,\eta)} \,d^{-D+sp}(a,\xi)\,d^{D+sp}(a,\eta)\,\mathrm{d}\nu(\xi).
	\end{multline}
	First, on $W=Z\setminus(U(a)\cup U(\eta))$, we use \eqref{localeq:ZminusUs-updown-bounds}, $sp<D$, $\sigma p<D$, and Lemma \ref{lem:Sobolev-for-a-ball}:
 \begin{multline*}
		\int_W\frac{\left|1-\big[\frac{ d(a,\xi)}{ d(a,\eta)}\big]^{\sigma+it}\right|^p}{ d_a^{D+sp}(\xi,\eta)}\,\mathrm{d}\nu_a(\xi)\\
		\prec_{\sigma,s,p}
		\int_{Z\setminus U(a)}\frac{\mathrm{d}\nu(\xi)}{d^{D-sp}(a,\xi)}+
		d^{D+sp-\sigma p}(a,\eta)\int_{Z\setminus U(a)}\frac{\mathrm{d}\nu(\xi)}{ d^{2D-\sigma p}(a,\xi)}\\
		\prec_{\sigma,s,p} d^{sp}(a,\eta).
	\end{multline*}
	On $U(a)$, we use \eqref{localeq:Ua-updown-bounds} and Lemma \ref{lem:int-over-ball-calculation} (or Remark \ref{rem:int-over-ball-calculation}):
	\begin{align*}
		\int_{U(a)}&\frac{\left|1-\big[\frac{ d(a,\xi)}{ d(a,\eta)}\big]^{\sigma+it}\right|^p}{ d_a^{D+sp}(\xi,\eta)}\,\mathrm{d}\nu_a(\xi)
		\prec_{s,p}
		\int_{U(a)}\frac{\left|1-\big[\frac{ d(a,\xi)}{ d(a,\eta)}\big]^{\sigma+it}\right|^p}{ d^{D-sp}(a,\xi)} \,\mathrm{d}\nu(\xi)\\
		&\prec_p\int_{U(a)}\frac{1}{d^{D-sp}(a,\xi)}\,\mathrm{d}\nu(\xi) +d^{-\sigma p}(a,\eta)\int_{U(a)}\frac{1}{d^{D-(s+\sigma)p}(a,\xi)}\,\mathrm{d}\nu(\xi)\\
		&\prec_{s,p,\sigma} d^{sp}(a,\eta).
	\end{align*}
	We now deal with $U(\eta)$.
	Using the Mean Value Theorem as in the proof of Proposition \ref{prop:geometric-control-1}, \eqref{localeq:Ueta-updown-bounds}, \eqref{localeq:1-d-over-d-triangle}, and Lemma \ref{lem:int-over-ball-calculation} (noting that $1-s>0$), we obtain
	\begin{align*}
		\int_{U(\eta)}&\frac{\left|1-\big[\frac{ d(a,\xi)}{ d(a,\eta)}\big]^{\sigma+it}\right|^p}{ d_a^{D+sp}(\xi,\eta)}\,\mathrm{d}\nu_a(\xi)\\
		&\prec_{\sigma,t} \int_{U(\eta)}\frac{\left|1-\frac{ d(a,\xi)}{ d(a,\eta)}\right|^p}{ d^{D+sp}(\xi,\eta)} \,d^{-D+sp}(a,\xi)\,d^{D+sp}(a,\eta)\,\mathrm{d}\nu(\xi)\\
		&\leq d^{2sp-p}(a,\eta) \int_{U(\eta)} \frac{\mathrm{d}\nu(\xi)}{d^{D-(1-s)p}(\xi,\eta)}
		\asymp_{\sigma,s,p} d^{sp}(a,\eta). \qedhere
	\end{align*}
\end{proof}

\section{The Cayley transform: Proof of Theorem \ref{thm:Cayley-in-intro}}\label{sect:Cayley}

In this section, we complete the proof of Theorem \ref{thm:Cayley-in-intro} in the Introduction; the two bounds are proved in Theorems \ref{thm:Cayley-op-is-bounded} and \ref{thm:Cayley-inverse-is-bounded} below.

\begin{thm}\label{thm:Cayley-op-is-bounded}
	Let $(Z,d,\nu)$ be a bounded Ahlfors $D$-regular metric measure space, and for each $a\in Z$ we consider $(Z_a,d_a,\nu_a)$ as defined in \eqref{eq:Z_a-d_a-nu_a-defn}.
	Let $1\le p<\infty$, $t\in\mathbb{R}$, and $s\in (0,1)$ with $sp< D$. Denote $z=s+it$.
	The operator
	\begin{align*}
		\Omega_{z,p}(a)&:W^{s,p}(Z)\subset L^p(Z)\rightarrow \dt{W}^{s,p}(Z_a),\\
		\Omega_{z,p}(a)\phi(\xi) &= d^{\frac{2D}{p}-2z}(a,\xi)\phi(\xi)\quad\text{for }\phi\in W^{s,p}(Z), \xi\in Z_a.
	\end{align*}
	is well defined and bounded, uniformly in $a\in Z$ and $t\in\mathbb{R}$.
\end{thm}
\begin{proof}
	Take $\phi\in W^{s,p}(Z)$. We calculate
	\begin{align*}
		\left\|\Omega_{z,p}(a)\phi | \dt{W}^{s,p}(Z_a)\right\|^p
		&=\iint_{Z^2_a}\frac{\left|\Omega_{z,p}(a)\phi(\xi)-\Omega_{z,p}(a)\phi(\eta)\right|^p}{d^{D+sp}_a(\xi,\eta)}\mathrm{d}\nu_a(\xi)\mathrm{d}\nu_a(\eta)\\
		&\stackrel{\eqref{eq:C_a-deriv-formula}}{=} \iint_{Z^2_a}\frac{\left|d^{\frac{2D}{p}-2z}(\xi,a)\phi(\xi)-d^{\frac{2D}{p}-2z}(\eta,a)\phi(\eta)\right|^p}{d^{D+sp}_a(\xi,\eta)}\mathrm{d}\nu_a(\xi)\mathrm{d}\nu_a(\eta)\\
		&\stackrel{\eqref{eq:Z_a-d_a-nu_a-defn}}{=} \iint_{Z^2}\frac{\left| d^{2\left(\frac{D}{p}-z\right)}(\xi,a)\phi(\xi)- d^{2\left(\frac{D}{p}-z\right)}(\eta,a)\phi(\eta)\right|^p}{ d^{D+sp}(\xi,\eta)}\cdot\\
		&\qquad\qquad\qquad\cdot \frac{1}{ d^{\left(\frac{D}{p}-s\right)p}(\xi,a)}\frac{1}{ d^{\left(\frac{D}{p}-s\right)p}(\eta,a)}\mathrm{d}\nu(\xi)\mathrm{d}\nu(\eta)\\
		&= \iint_{Z^2}\frac{\left| \frac{ d^{{D}/{p}-z}(\xi,a)}{ d^{{D}/{p}-z}(\eta,a)}\phi(\xi)- \frac{ d^{{D}/{p}-z}(\eta,a)}{ d^{{D}/{p}-z}(\xi,a)}\phi(\eta)\right|^p}{ d^{D+sp}(\xi,\eta)} \mathrm{d}\nu(\xi)\mathrm{d}\nu(\eta).
	\end{align*}
	Next, using the estimate: $|x+y+z|^p\leq 3^{p-1}(|x|^p+|y|^p+|z|^p)$ for all $x,y,z\in\BC$, we further have
	\begin{multline*}
		\iint_{Z^2}\frac{\left| \frac{ d^{{D}/{p}-z}(\xi,a)}{ d^{{D}/{p}-z}(\eta,a)}\phi(\xi)- \frac{ d^{{D}/{p}-z}(\eta,a)}{ d^{{D}/{p}-z}(\xi,a)}\phi(\eta)\right|^p}{ d^{D+sp}(\xi,\eta)} \mathrm{d}\nu(\xi)\mathrm{d}\nu(\eta)\\
		\leq 3^{p-1}\left\|\phi| W^{s,p}(Z) \right\|^p
		+2\cdot 3^{p-1}\int_{Z}\,\int_Z\frac{\left|1-\frac{ d^{{D}/{p}-z}(\eta,a)}{ d^{{D}/{p}-z}(\xi,a)}\right|^p}{ d^{D+sp}(\xi,\eta)} \mathrm{d}\nu(\xi)\cdot|\phi(\eta)|^p \mathrm{d}\nu(\eta).
	\end{multline*}
	Using Lorentz--H\"older inequality (Proposition \ref{prop:Lorentz-Holder-ineq}), Lemma \ref{lem:Lorentz-powers}, the geometric control Proposition \ref{prop:geometric-control-1} (as $p\cdot\left(D/p-s\right)<D$), and the fractional Sobolev inequality (Theorem \ref{thm:Sobolev-inequality}), we have
	\begin{align*}
		\int_{Z}\,\int_Z&\frac{\left|1-\frac{ d^{{D}/{p}-z}(\eta,a)}{ d^{{D}/{p}-z}(\xi,a)}\right|^p}{ d^{D+sp}(\xi,\eta)} \mathrm{d}\nu(\xi)\cdot|\phi(\eta)|^p \mathrm{d}\nu(\eta)\\
		&\leq \left\|\int_Z \frac{\left|1-\frac{ d^{{D}/{p}-z}(\cdot,a)}{ d^{{D}/{p}-z}(\xi,a)}\right|^p}{ d^{D+sp}(\xi,\cdot)}d\xi \left|L\left(\tfrac{D}{sp},\infty\right)\right.\right\|\cdot\left\||\phi|^p\left|L\left(\tfrac{D}{D-sp},1\right)\right.\right\|\\
		&\prec_{z,p} \left\|\phi\left|L\left(\tfrac{pD}{D-sp},p\right)\right.\right\|^p
		\prec_{s,p}\left\|\phi|W^{s,p}(Z)\right\|^p.
	\end{align*}
	To finish the proof, we need to show that $\Omega_{z,p}(a)\phi$ is a limit (in $[\cdot]_{s,p}$) of compactly supported functions on $Z_a$. First, we can assume that $\phi$ (and hence also $\Omega_{z,p}(a)\phi$) is a bounded function, by considering cut-offs of $\phi$, by the argument at the beginning of the proof of \cite[Theorem 6.5]{Sobolev-Hitchhikers}. Next, we prove the following statement: if $\psi:Z_a\to\C$ is a bounded measurable function with $[\psi]_{s,p}<\infty$, then it is a $[\cdot]_{s,p}$-limit of boundedly supported ones. For this, fix $o\in Z_a$, and consider functions $f_R:Z_a\to[0,1]$ for $R>0$, such that $f_R|_{B(o,R)}\equiv 1$, $f_R|_{Z_a\setminus B(o,R+1)}\equiv 0$, and $f_R$ are all $K$-Lipschitz for some $C\geq 0$. For example, one can take
	\begin{equation*}
		f_R(\xi) = \frac{d_a(\xi,Z_a\setminus B(o,R+1))}{d_a\!\left(\xi,\overline{B(o,R)}\right)+d_a(\xi,Z_a\setminus   B(o,R+1))},
		\quad \xi\in Z_a.
	\end{equation*}
	Observe that for $\xi,\eta\in Z_a$, we have
	\begin{equation*}
		\left| \psi(\xi)f_R(\xi)  - \psi(\eta)f_R(\eta)\right| \leq |\psi(\xi)-\phi(\eta)| + K\|\psi|L^\infty\| \cdot d_a(\xi,\eta).
	\end{equation*}
	Consequently (cf.~Remark \ref{rem:Wsp-contains-Lip-and-regularity}),
	\begin{equation}\label{localeq:psi-times-f_R-bound}
		[\psi\cdot f_R]_{s,p} \leq [\psi]_{s,p} + C\|\psi | L^\infty\|<\infty,
	\end{equation}
	independently of $R>0$. Consider now functions $F_R: Z_a\times Z_a\to \mathbb{C}$, defined by
	\begin{equation*}
		F_R(\xi,\eta) = (\psi\cdot f_R - \psi)(\xi) - (\psi\cdot f_R - \psi)(\eta), \quad \xi,\eta\in Z_a.
	\end{equation*}
	Then $F_R\to 0$ pointwise for $R\to\infty$. Furthermore $|F_R(\xi,\eta)| \leq 2|\psi(\xi)-\phi(\eta)| + K\|\psi|L^\infty\| \cdot d_a(\xi,\eta)$, and the function on the right-hand side belongs to $L^p\!\left(Z_a\times Z_a, \frac{\mathrm{d}\nu_a\mathrm{d}\nu_a}{d^{D+sp}}\right)$ (similarly to \eqref{localeq:psi-times-f_R-bound} above). Hence, by Lebesgue's Dominated Convergence Theorem, $F_R\to 0$ in $L^p\!\left(Z_a\times Z_a, \frac{\mathrm{d}\nu_a\mathrm{d}\nu_a}{d^{D+sp}}\right)$; in other words, $\psi\cdot f_R \to \psi$ in $[\cdot]_{s,p}$.
\end{proof}

\begin{thm}\label{thm:Cayley-inverse-is-bounded}
	With the same assumptions and notation as in Theorem \ref{thm:Cayley-op-is-bounded},
	the operator
	\begin{align*}
	\Omega^{-1}_{z,p}(a) &: \dt{W}^{s,p}(Z_a) \rightarrow W^{s,p}(Z)\subset L^p(Z)\\
	\Omega^{-1}_{z,p}(a)\phi &= \left[ d^{-\frac{2D}{p}+2z}(a,\cdot)\phi\right] \in L^p(Z)
	\end{align*}
	is well defined and bounded, uniformly in $a\in Z$ and $t\in\mathbb{R}$.
\end{thm}

\begin{proof}
	Let $\phi\in \dt{W}^{s,p}(Z_a)$ be measurable and compactly supported.

	Using Lorentz--H\"older inequality (Proposition \ref{prop:Lorentz-Holder-ineq}), Lemma \ref{lem:Lorentz-powers}, Lemma \ref{lem:d^kalpha-in-weak-L-Z_a-D/beta}, and the fractional Sobolev inequality (Theorem \ref{thm:Sobolev-inequality}):
	\begin{align*}
		\left\|\Omega^{-1}_{z,p}(a)\phi\,|L^p(Z)\right\|^p& =
		\int_Z\left| d^{-\frac{2D}{p}+2z}(\xi,a)\phi(\xi)\right|^{p}\mathrm{d}\nu(\xi)
		\stackrel{\eqref{eq:Z_a-d_a-nu_a-defn}}{=} \int_{Z_a} d^{2sp}(\xi,a)|\phi(\xi)|^p\mathrm{d}\nu_a(\xi)\\
		&\leq \left\| d^{2sp}(\cdot,a)\left|L\left(Z_a;\tfrac{D}{sp},\infty\right)\right.\right\|\cdot\left\|\phi\left|L\left(Z_a;\tfrac{pD}{D-sp},p\right)\right.\right\|^p\\
		&\prec_{s,p} \left\|\phi\left|\dt{W}^{s,p}(Z_a)\right.\right\|^p.
	\end{align*}
	To bound the second term in $W^{s,p}(Z)$-norm, we calculate as in the proof of Theorem \ref{thm:Cayley-op-is-bounded}:
	\begin{align*}
	&\iint_{Z^2}\frac{\left|\Omega_{z,p}^{-1}(a)\phi(\xi)-\Omega^{-1}_{z,p}(a)\phi(\eta)\right|^p}{d^{D+sp}(\xi,\eta)}\mathrm{d}\nu(\xi)\mathrm{d}\nu(\eta)\\
	&\stackrel{\eqref{eq:C_a-deriv-formula}}{=} \iint_{Z^2}\frac{\left|d^{-\frac{2D}{p}+2z}(\xi,a)\phi(\xi)-d^{-\frac{2D}{p}+2z}(\eta,a)\phi(\eta)\right|^p}{d^{D+sp}(\xi,\eta)}\mathrm{d}\nu(\xi)\mathrm{d}\nu(\eta)\\
	&\stackrel{\eqref{eq:Z_a-d_a-nu_a-defn}}{=} \iint_{Z_a^2}\frac{\left| d^{-2\left(\frac{D}{p}+z\right)}(\xi,a)\phi(\xi)- d^{-2\left(\frac{D}{p}+z\right)}(\eta,a)\phi(\eta)\right|^p}{ d_a^{D+sp}(\xi,\eta)}\cdot\\
	&\qquad\qquad\qquad\cdot d^{\left(\frac{D}{p}-s\right)p}(\xi,a) d^{\left(\frac{D}{p}-s\right)p}(\eta,a)\,\mathrm{d}\nu_a(\xi)\mathrm{d}\nu_a(\eta)\\
	&= \iint_{Z_a^2}\frac{\left| \frac{ d^{{D}/{p}-z}(\eta,a)}{ d^{{D}/{p}-z}(\xi,a)}\phi(\xi)- \frac{ d^{{D}/{p}-z}(\xi,a)}{ d^{{D}/{p}-z}(\eta,a)}\phi(\eta)\right|^p}{ d_a^{D+sp}(\xi,\eta)}
	\mathrm{d}\nu_a(\xi)\,\mathrm{d}\nu_a(\eta)\\
	&\leq 3^{p-1}\left\|\phi \left|\dt{W}^{s,p}(Z_a)\right.\right\|^p + 2\cdot 3^{p-1} \int_{Z}\,\int_Z\frac{\left|1-\frac{ d^{{D}/{p}-z}(\xi,a)}{ d^{{D}/{p}-z}(\eta,a)}\right|^p}{ d_a^{D+sp}(\xi,\eta)} \mathrm{d}\nu_a(\xi)\cdot|\phi(\eta)|^p \mathrm{d}\nu_a(\eta)\\
	&\prec_{z,p} \left\|\phi \left|\dt{W}^{s,p}(Z_a)\right.\right\|^p,
	\end{align*}
	where the last inequality follows the same Lorentz norm estimate as in the proof of Theorem \ref{thm:Cayley-op-is-bounded}, now using the other geometric control Proposition \ref{prop:geometric-control-2}.
\end{proof}

\section{Representations: Proof of Theorem \ref{thm:UB-reps-in-intro}}
\label{sect:representations}

In this Section, we deduce Theorem \ref{thm:UB-reps-in-intro} from Theorem \ref{thm:Cayley-in-intro} much like we explained in Section \ref{sect:outline}.

Let us recall the setup of Theorem \ref{thm:UB-reps-in-intro}. Let $(Z,d)$ be a complete metric space without isolated points, let $\nu$ be the $D$-dimensional Hausdorff measure on $Z$, and suppose that $(Z,d,\nu)$ is Ahlfors $D$-regular. Let $t\in \mathbb{R}$, $s\in(0,1)$ and $p\in [1,\infty)$ with $sp< D$. Denote $z=s+it$.
Recall that we are considering the bundle $(Z_a,d_a,\nu_a)_{a\in Z}$ defined by \eqref{eq:Z_a-d_a-nu_a-defn}.

For $a\in Z$, and $g\in \Mob(Z)$, we check that the operators $\Pi_w^{(p)}(g;a): W^{s,p}(Z_a)\to W^{s,p}(Z_{ga})$,  $\Pi^{(p)}_{z}(g;a)\phi = \D{g}^{\frac{D}{p}-z}(a)\, g^*\phi$, are isometries.
\begin{multline}\label{eq:Pi_z_p_g_a-is-isometric}
		\left\|\Pi_z^{(p)}(g;a)\phi\left| \dt{W}^{s,p}(Z_{ga}) \right.\right\|^p\\
	\begin{aligned}
	 	&\stackrel{\eqref{eq:defn-of-[]sp-intro}}{=}
		\iint_{Z_{ga}}\frac{\left|\Pi_z^{(p)}(g;a)\phi(\xi)-\Pi_z^{(p)}(g;a)\phi(\eta)\right|^p}{d_{ga}^{D+sp}(\xi\eta)}\,\mathrm{d}\nu_{ga}(\xi)\mathrm{d}\nu_{ga}(\eta)\\
		&= \D{g}^{D-s p}(a)\cdot \iint_{Z_{ga}}  \frac{\left|\phi(\omega)-\phi(\tau)\right|^p}{d_{ga}^{D+sp}(g\omega,g\tau)}\,\mathrm{d}\nu_{ga}(g\omega)\mathrm{d}\nu_{ga}(g\tau)\\
		&\stackrel{\eqref{eq:GMV-for-d_a}\eqref{eq:change-of-var-nu_a}}{=} \D{g}^{D-sp}(a)\cdot \iint_{Z_a} \frac{\left|\phi(\omega)-\phi(\tau)\right|^p}{d_{a}^{D+sp}(\omega,\tau)} \cdot \frac{\D{g}^{D+sp}(a)}{\D{g}^{2D}(a)}\,\mathrm{d}\nu_a(\omega)\mathrm{d}\nu_a(\tau)\\
		&= %
		\left\|\phi\left| \dt{W}^{s,p}(Z_{a})\right.\right\|^p.
	\end{aligned}
\end{multline}
Furthermore, they satisfy the relation of being an (isometric) representation of $\Mob(Z)\ltimes Z$, namely that $\Pi_z^{(p)}(gh;a) = \Pi_z^{(p)}(g;ha)\circ \Pi_z^{(p)}(h;a)$ for $a\in Z$, $g,h\in\Mob(Z)$. Indeed, for $\phi\in \dt{W}^{s,p}(Z_a)$, $\xi\in Z_{gha}$ we have
\begin{gather}\label{eq:pi_z_g_a-is-a-groupoid-repn}
\begin{aligned}
	\Pi_z^{(p)}(gh;a)\phi(\xi) &= \Db{gh}^{\frac{D}{p}-z}(a)\phi(h^{-1}g^{-1}\xi) \\
	&\stackrel{\eqref{eq:metric-deriv-cocycle}}{=}
	\D{g}^{\frac{D}{p}-z}(ha)\,\D{h}^{\frac{D}{p}-z}(a)\phi(h^{-1}g^{-1}\xi)\\
	&=\D{g}^{\frac{D}{p}-z}(ha)\,\Pi_z^{(p)}(h;a)\,\phi(g^{-1}\xi)\\
	&= \Pi_z^{(p)}(g;ha)\,\Pi_z^{(p)}(h;a)\,\phi(\xi).
\end{aligned}
\end{gather}
We now calculate
for $g\in\Mob(Z)$, $\phi\in W^{s,p}(Z)$, and $a,\xi\in Z$,
\begin{multline}\label{calc:pi_z-via-Omega-is-as-wanted}
		\Omega^{-1}_{z,p}(ga)\Pi_{z}^{(p)}(g;a)\,\Omega_{z,p}(a)\,\phi(\xi)\\
	\begin{aligned}
		&\stackrel{\eqref{eq:C_a-deriv-formula}\eqref{eq:Pi_s_g_a-defn}}{=}
		d^{-\frac{2D}{p}+2z}(\xi,ga) \cdot \D{g}^{\frac{D}{p}-z}(a) \cdot d^{\frac{2D}{p}-2z}(g^{-1}\xi,a) \cdot \phi(g^{-1}\xi) \\
		&\stackrel{\eqref{eq:GMV}}{=} d^{-\frac{2D}{p}+2z+\frac{2D}{p}-2z}(g^{-1}\xi,a) \cdot
		\D{g}^{-\frac{D}{p}+z}(g^{-1}\xi) \cdot \phi(g^{-1}\xi) \\
		&\stackrel{\eqref{eq:metric-deriv-cocycle}}{=}
		\Db{g^{-1}}^{\frac{D}{p}-z}(\xi) \cdot \phi(g^{-1}\xi)
		=
		\pi_z^{(p)}(g)\phi(\xi).
	\end{aligned}
\end{multline}
We see that we recover our boundary representation as
$\pi_z^{(p)}(g) = \Omega_{z,p}^{-1}(ga)\circ\Pi_z^{(p)}(g;a)\circ\Omega_{z,p}(a)$,
showing that it is uniformly bounded on $W^{s,p}(Z)$ by Theorem \ref{thm:Cayley-in-intro}. 
This finishes the proof of Theorem \ref{thm:UB-reps-in-intro}.

\begin{rem}
	The calculation \eqref{calc:pi_z-via-Omega-is-as-wanted} also implies that the cocycle $c_{z,p}(a,b) = \Omega_{z,p}(b)\circ \Omega^{-1}_{z,p}(a)$ (for $a,b\in Z$) is covariant for the representation $\Pi_z^{(p)}$ of $\Mob(Z)\ltimes Z$, i.e.~the relation \eqref{eq:ub-cocycle-equivariant}, as the final expression in \eqref{calc:pi_z-via-Omega-is-as-wanted} then does not depend on $a$. Consequently, we may view $\pi_z^{(p)}$ as conjugate (by the operators $\Omega_{z,p}$) to uniformly bounded representations of $\Mob(Z)$ on $W^{s.p}(Z_a)$, for any $a\in Z$, given by $g\mapsto c_{z,p}(a,ga)\Pi_z^{(p)}(g;a)$. This is analogous to the ``noncompact picture'' for representations of Lie groups.
\end{rem}

\begin{rem}
	By Proposition \ref{prop:Wsp-is-a-closed-subsp-of-L^p}, the (uniformly bounded) representations $\pi_{z}^{(p)}$ of $\Mob(Z)$ from Theorem \ref{thm:UB-reps-in-intro} are conjugate to uniformly bounded representations on closed subspaces of an $L^p$-space.
\end{rem}

\begin{rem}
	Continuing with the setup and notation of Theorem \ref{thm:UB-reps-in-intro},
	the fractional Sobolev Inequality (Theorem \ref{thm:Sobolev-inequality}) shows that the spaces $W^{s,p}(Z)$ are continuously embedded in the Lorentz function spaces $L(Z;p^*,p)$, with $p^*=\frac{pD}{D-sp}$. The dual of  $L(Z;p^*,p)$ (with respect to the $L^2$-pairing) is another Lorentz space, $L\!\left(Z;\frac{qD}{D+sq},q\right)$, where $1/p+1/q=1$, see e.g.~\cite[Theorem 1.4.16(vi)]{Grafakos:2014}. Consequently, the dual of $W^{s,p}(Z)$ contains this Lorentz space as a dense subspace, and straightforward calculations show that the ``dual'' representations $\pi_{-\overline{z}}^{(q)}$ are uniformly bounded on the dual of $W^{s,p}(Z)$ (cf.~\cite[p.53, p.102]{Figa-Talamanca-Picardello:book:1983}).
\end{rem}

\section{Almost invariant vectors: Proof of Theorem \ref{thm:almost-inv-subspace-Lp-in-intro}}
\label{sect:almost-invariant-vectors}

We start with an equicontinuity Lemma, which follows from the proof of \cite[Lemma 7]{Nica:2013}, and \cite[Proposition 14]{Nica:2013}.

\begin{lem}\label{locallem:Dg-from-K-are-uniformly-Lipschitz}
	Let $(Z,d)$ be a complete metric space without isolated points.
	Let $K\subset \Mob(Z)$ be a compact subset of $\Mob(Z)$. Then there exists $\lambda>0$ such that
	all functions $\D{g}$ for $g\in K$ are $\lambda$-Lipschitz.
\end{lem}

Next, we show one part of Theorem \ref{thm:almost-inv-subspace-Lp-in-intro}, namely that $\pi_s^{(p)}$ have almost invariant vectors (for an appropriate $s$).

\begin{prop}
	Let $(Z,d,\nu)$ be a bounded Ahlfors $D$-regular metric space, where $\nu$ is the Hausdorff $D$-measure on $(Z,d)$.
	Let $K\subset \Mob(Z)$ be a compact set, $\varepsilon>0$, and $p>\max(D,1)$.
	Then there exists $s_0\in (0,\frac{D}{p})$ such that for all $s\in (s_0,\frac{D}{p})$ we have
	\begin{equation*}
		\sup_{g\in K}\left\|\pi_s^{(p)}(g)\constone-\constone\mid W^{s,p}(Z)\right\|\leq \varepsilon.
	\end{equation*}
\end{prop}
\begin{proof}
	Take a compact $K\subset \Mob(Z)$.
	Then there exists $c>0$ such that 
	$c^{-1}\constone \leq \Db{g^{-1}} \leq c\constone$ for all $g\in K$.
	It follows that $\Db{g^{-1}}^{\frac{D}{p}-s}\xrightarrow{s\rightarrow \frac{D}{p}} \constone$ uniformly on $Z$ and $K$. As a consequence
	\begin{align*}
		\left\|\pi_{s}^{(p)}(g)\constone-\constone\mid L^p(Z)\right\|^p&=\int_Z \left|\Db{g^{-1}}^{\frac{D}{p}-s}(\xi)-1\right|^p \mathrm{d}\nu(\xi)\xrightarrow{s\rightarrow \frac{D}{p}}0.
	\end{align*}
	This deals with the $L^p$-part of the norm on $W^{s,p}(Z)$.
	
	Denote $\Phi(t)=t^{\frac{D}{p}-s}$ for $t>0$. Using the Mean Value Theorem and Lemma \ref{locallem:Dg-from-K-are-uniformly-Lipschitz}, we estimate for $\xi,\eta\in Z$:
	\begin{multline*}
		\left|\Db{g^{-1}}^{\frac{D}{p}-s}(\xi)-\Db{g^{-1}}^{\frac{D}{p}-s}(\eta)\right|\\
		\begin{aligned}
			&=\left|\Phi\!\left(\Db{g^{-1}}(\xi)\right)-\Phi\!\left(\Db{g^{-1}}(\eta)\right)\right|\\
			&\leq \left({\tfrac{D}{p}-s}\right)\cdot \max_{t\in[c^{-1},c]}\left(t^{\frac{D}{p}-s-1}\right)\cdot \left|\Db{g^{-1}}(\xi)-\Db{g^{-1}}(\eta)\right|\\
			&\leq \left({\tfrac{D}{p}-s}\right)\cdot \max_{t\in[c^{-1},c]}\left(t^{\frac{D}{p}-s-1}\right)\cdot \lambda \cdot d(\xi,\eta)\\
			&\prec_{K,p} \left({\tfrac{D}{p}-s}\right)\cdot d(\xi,\eta).
		\end{aligned}
	\end{multline*}
	Consequently,
	\begin{align*}
		\left[\pi_{s}^{(p)}(g)\constone-\constone\right]_{s,p}&=\iint \frac{\left|\Db{g^{-1}}^{\frac{D}{p}-s}(\xi)-\Db{g^{-1}}^{\frac{D}{p}-s}(\eta)\right|^p}{ d^{D+sp}(\xi,\eta)}
		\mathrm{d}\nu(\xi)\mathrm{d}\nu(\eta)\\
		&\prec_{K,p}\left({\tfrac{D}{p}-s}\right)^p\iint \frac{1}{ d^{D+sp-p}(\xi,\eta)} \mathrm{d}\nu(\xi)\mathrm{d}\nu(\eta)\\
		&\leq\left({\tfrac{D}{p}-s}\right)^p\iint \frac{1}{ d^{D-(p-D)}(\xi,\eta)} \mathrm{d}\nu(\xi)\mathrm{d}\nu(\eta)
		\xrightarrow{s\rightarrow \frac{D}{p}}0,
	\end{align*}
	by Lemma \ref{lem:int-over-ball-calculation}, as $p-D>0$.
\end{proof}

The second part of Theorem \ref{thm:almost-inv-subspace-Lp-in-intro} is that $\pi_s^{(p)}$ have no non-zero invariant vectors. For this we use amenability of the boundary action.

\begin{defn}
An action of a countable discrete group $\Gamma$ on a compact Hausdorff space $Y$ is (topologically) \emph{amenable}, if there exists a sequence of continuous\footnote{We think of $\tprob(\Gamma)$ as a convex subset of $\ell^1(\Gamma)$, endowed with the weak-* topology.} maps
$f_n:Y\rightarrow \tprob(\Gamma)$ such that for every $g\in\Gamma$
$$
\lim_{n\to\infty}\sup_{y\in Y}\|f_n(g^{-1}y)-g^*(f_n(y))\|_1 = 0.
$$
\end{defn}

The following is classical, cf.~e.g.~\cite[Chapter 5, Exercise 4.4]{Brown-Ozawa:2008}.

\begin{prop}\label{prop:amenable-action-invariant-measure}
Let $\Gamma \curvearrowright Y$ be an amenable action of a discrete group $\Gamma$ on a compact space $Y$.
Then $\Gamma$ is amenable if (and only if) $Y$ admits a $\Gamma$-invariant probability measure.
\end{prop}

Boundary actions of hyperbolic groups are amenable: Adams proved these actions are measurewise amenable (for any quasi-invariant measure on the boundary) \cite{Adams:1994}, and subsequently Anantharaman--Delaroche and Renault \cite[Chapter 3]{Anantharaman-Delaroche-Renault:2000} shown that this implies topological amenability in the above sense (see also \cite[Appendix B]{Anantharaman-Delaroche-Renault:2000}).

\begin{thm*}[Adams]
The action of a countable discrete hyperbolic group $\Gamma$ on its Gromov boundary is amenable.
\end{thm*}

\begin{prop}
Let $(Z,d,\nu)$ be an Ahlfors $D$-regular metric measure space. Let $\Gamma<\Mob(Z)$ be a non-amenable group, such that the action of $\Gamma \curvearrowright Z$  is (topologically) amenable. Then $\pi^{(p)}_s$ does not have non-zero invariant vectors in $W^{s,p}(Z)$.
\end{prop}
\begin{proof}
We proceed by contradiction: suppose that there exists $0\neq\psi\in W^{s,p}(Z)$ which is $\pi^{(p)}_s$-invariant. 
Then $A=\{|\psi|\neq0\}$ is a $\Gamma$-invariant non-empty subset of $Z$.

Denoting $p^*=\frac{pD}{D-sp}$, the classical Sobolev inequality (i.e.~Theorem \ref{thm:Sobolev-inequality} combined with the inclusion $L(Z;p^*,p)\subset L^{p^*}\!(Z)$) implies
that $\psi\in L^{p^*}\!(Z)$. Furthermore $\|\psi | L^{p^*}\|\not=0$.

Expanding $\pi^{(p)}_s(g)\psi=\psi$, we have
\begin{align*}
\Db{g^{-1}}^{\frac{D}{p}-s}\psi(g^{-1}\xi) &= \psi(\xi)\\
\iff \Db{g^{-1}}^{D}|\psi|^{p^*}\!(g^{-1}\xi)&=|\psi|^{p^*}\!(\xi)\\
\iff g^*(|\psi|^{p^*}\!(\xi)\mathrm{d}\nu(\xi))&=|\psi|^{p^*}\!(\xi)\mathrm{d}\nu(\xi)
\end{align*}
for all $g\in\Gamma$ and for $\nu$-almost every $\xi\in Z$.
In other words, the probability measure $\frac{1}{\|\psi|L^{p^*}\|}|\psi|^{p^*}\mathrm{d}\nu$ on $Z$ is $\Gamma$-invariant. This is a contradiction with Proposition \ref{prop:amenable-action-invariant-measure}.
\end{proof}

The final part of Theorem \ref{thm:almost-inv-subspace-Lp-in-intro} is a bound on $\|\pi_{s}^{(p)}(g)\|$ independent of $s\in(0,\frac{D}{p})$, but allowed to depend on $g\in\Gamma$. This is provided by the following Proposition, where we use the notation $\|\psi\|_\infty$ for the $L^{\infty}$-norm of $\psi$.
\begin{prop}
	Let $(Z,d,\nu)$ be an Ahlfors $D$-regular metric measure space. 
    For $p\in (1,\infty)$ with $p>D$, there exist $C,A>0$ such that:
    $$
    \|\pi_s^{(p)}(g)\|_{W^{s,p}(Z)\to W^{s,p}(Z)}\leq C\left\|\D{g}\right\|_\infty^A \left\|\Db{g^{-1}}\right\|_\infty^A
    $$
    for every $s\in(0,\frac{D}{p})$ and $g\in \Mob(Z)$.
\end{prop}

\begin{proof}
Let $g\in\Mob(Z)$. We start by observing that $\min_{\xi\in Z}\D{g}(\xi) = \|\Db{g^{-1}}\|_\infty$ by \eqref{eq:metric-deriv-cocycle}, and hence by the proof of \cite[Lemma 7]{Nica:2013}, the Lipschitz constant of $\D{g}^{1/2}$ is bounded by a constant multiple of $\|\D{g}\|_\infty\cdot \|\Db{g^{-1}}\|_\infty^{1/2}$ . Applying the Mean Value Theorem and \eqref{eq:metric-deriv-cocycle} again implies that the Lipschitz constant $L$ of the function $\D{g}^{-\frac{D}{p}+s}(\cdot)$ is bounded by a constant multiple of $\|\D{g}\|_\infty^2\cdot \|\Db{g^{-1}}\|_\infty$ (remember $s$ has bounded range).

Take $\phi\in W^{s,p}(Z)$, and calculate:
\begin{align*}
\left[\pi_s^{(p)}(g)\phi\right]_{s,p}^p
&=\iint_{Z}\frac{\left|\Db{g^{-1}}^{\frac{D}{p}-s}(\xi)\phi(g^{-1}\xi)-\Db{g^{-1}}^{\frac{D}{p}-s}(\eta)\phi(g^{-1}\eta)\right|^p}{d^{D+sp}(\xi,\eta)}\mathrm{d}\nu(\xi)\mathrm{d}\nu(\eta)\\
\intertext{(using \eqref{eq:metric-deriv-cocycle}, and denoting $\omega=g^{-1}\xi$, $\tau=g^{-1}\eta$)}
&=\iint_{Z}\frac{\left|\D{g}^{-\frac{D}{p}+s}(\omega)\phi(\omega)-\D{g}^{-\frac{D}{p}+s}(\tau)\phi(\tau)\right|^p}{d^{D+sp}(g\omega,g\tau)}\mathrm{d}\nu(g\omega)\mathrm{d}\nu(g\tau)\\
\intertext{(applying \eqref{eq:GMV},\eqref{eq:change-of-var}, and bounding by the supremum of $\D{g}$)}
&\leq \|\D{g}\|_\infty^{D-sp}
\iint_{Z}\frac{\left|\D{g}^{-\frac{D}{p}+s}(\omega)\phi(\omega)-\D{g}^{-\frac{D}{p}+s}(\tau)\phi(\tau)\right|^p}{d^{D+sp}(\omega,\tau)}\mathrm{d}\nu(\omega)\mathrm{d}\nu(\tau)\\
\intertext{(inserting $\pm \D{g}^{-\frac{D}{p}+s}(\tau)\phi(\omega)$, using $|a+b|^p\leq 2^{p-1}(|a|^p+|b|^p)$, and $s>0$)}
&\leq 2^{p-1}\|\D{g}\|_\infty^{D}
\int_{Z}|\phi(\omega)|^p\int_Z\frac{\left|\D{g}^{-\frac{D}{p}+s}(\omega)-\D{g}^{-\frac{D}{p}+s}(\tau)\right|^p}{d^{D+sp}(\omega,\tau)}\mathrm{d}\nu(\tau)\mathrm{d}\nu(\omega)\\
&\quad+2^{p-1}\|\D{g}\|_\infty^{D}\iint_Z\frac{\D{g}^{-D+sp}(\tau)\cdot |\phi(\omega)-\phi(\tau)|^p}{d^{D+sp}(\omega,\tau)}\mathrm{d}\nu(\omega)\mathrm{d}\nu(\tau)\\
\intertext{(using $0<s<\frac{D}{p}$; that $L$ is the Lipschitz constant of  $\D{g}^{-\frac{D}{p}+s}(\cdot)$; and  \eqref{eq:metric-deriv-cocycle} again)}
&\leq 2^{p-1}\|\D{g}\|_\infty^{D}
\,L^p\int_Z|\phi(\omega)|^p\int_Z\frac{1}{d^{2D-p}(\omega,\tau)}\mathrm{d}\nu(\tau)\mathrm{d}\nu(\omega)\\
&\quad+2^{p-1}\|\D{g}\|_\infty^{D}\|\Db{g^{-1}}\|_\infty^{D-sp}
\iint_Z\frac{\left|\phi(\omega)-\phi(\tau)\right|^p}{d^{D+sp}(\omega,\tau)}\mathrm{d}\nu(\omega)\mathrm{d}\nu(\tau)\\
\intertext{(using Lemma \ref{lem:int-over-ball-calculation} since $2D-p<D$  )}
&\leq C_1(p)\, 2^{p-1}\|\D{g}\|_\infty^{D}\,L^p\|\phi\|_p^p
+2^{p-1}\|\D{g}\|_\infty^{D}\|\Db{g^{-1}}\|_\infty^{D}
[\phi]_{s,p}^p\\
&\leq C_2(p) \||g'|\|_\infty^{D+2p}\|\Db{g^{-1}}\|_\infty^{p}\|\phi|W^{s,p}(Z)\|^p.
\end{align*}
To finish the proof, we estimate the $L^p$-part of the norm on $W^{s,p}(Z)$:
\begin{align*}
	\left\|\pi_s^{(p)}(g)\phi|L^p\right\|^p &= \int_{Z} \Db{g^{-1}}^{D-sp}(\xi)\left|\phi(g^{-1}\xi)\right|^p\mathrm{d}\nu(\xi)\\
	&\leq \left\| \Db{g^{-1}} \right\|_\infty^{D}
	  \int_Z | \phi(\omega) |^p \D{g}^D(\omega)\mathrm{d}\nu(\omega)\\
	 &\leq \left\| \Db{g^{-1}} \right\|_\infty^{D} \left\| \D{g} \right\|_\infty^{D} \left\| \phi | L^p \right\|^p.
	 \qedhere
\end{align*}
\end{proof}

\section{Potential theory and dual representations when $p=2$}\label{sect:potential-theory}

In this section, we prove Theorem \ref{thm:potentials-in-intro}.

For this section, let $(Z,d,\nu)$ be an Ahlfors $D$-regular metric measure space. Let $s\in(0,1)$.

Denote $\mathcal{E}_s=[\cdot]_{s,2}$. We want to consider $\mathcal{E}_s$ as a regular Dirichlet form on $L^2(Z)$ \cite{Fukushima-Oshima-Takeda:2011}; this is explained e.g.~in \cite[p.~251]{Hu-Kumagai:2006}. The only potential issue is the domain of $\mathcal{E}_s$. As mentioned in Section \ref{sect:sobolev-spaces}, the set $O=W^{s,2}(Z)\cap C_0(Z)$ contains compactly supported Lipschitz functions; the question is whether $O$ is dense in $W^{s,2}(Z)$ with respect to $\langle\cdot,\cdot\rangle_{L^2}+\mathcal{E}_s[\cdot]$. We \emph{define} $H_0^s(Z)=W_0^{s,2}(Z)$ to be the closure on $O$ in this norm, and consider the domain of $\mathcal{E}_s$ to be $H_0^s(Z)$ with core $O$.\footnote{During the process of writing this piece, we were informed by G.~Gerontogiannis (private communication) that in fact $H_0^s(Z) = H^s(Z)$.}

\begin{rem}\label{rem:pi_s-restricts-to-H_0}
	Assume for a moment that $\nu$ is the Hausdorff $D$-measure on $(Z,d)$.
	Then since $\D{g}\in\Lip(Z)$ for any $g\in \Mob(Z)$, $\pi_s^{(2)}$ preserves the core $O$, and thus extends to a uniformly bounded representation of $\Mob(Z)$ on $H_0^s(Z)$ by Theorem \ref{thm:UB-reps-in-intro} (which we continue to denote $\pi_s^{(2)}$).
\end{rem}

By general theory \cite{Fukushima-Oshima-Takeda:2011}, $\mathcal{E}_s$ has an infinitesimal generator that we will denote $-\Delta^s$, i.e.~a non-positive self-adjoint operator with domain $\mathcal{D}(\Delta^s)\subset H^s_0(Z)$ such that $\mathcal{E}_s[\phi]=\langle \sqrt{\Delta^s}\phi, \sqrt{\Delta^s}\phi\rangle$ for $\phi\in H^s_0(Z)$. Furthermore, there is a (strongly continuous) heat semigroup $T_t=\exp(-t\Delta^s)$ for $t>0$.  Potential theory now tells us that $T_t$ are kernel operators, and provides a \emph{stable-like} estimate on the kernels:
\begin{align*}
T_t\phi(\eta) &= \int_Z p_t(\xi,\eta)\phi(\xi)\mathrm{d}\nu(\xi), \quad \phi\in L^2(Z),\, \eta\in Z,\\
p_t(\xi,\eta) &\asymp \frac{1}{t^{D/(2s)}}\left(1+\frac{ d(\xi,\eta)}{t^{1/(2s)}}\right)^{-(D+2s)}
\asymp \min\left(t^{-\frac{D}{2s}},\frac{t}{d^{D+2s}(\xi,\eta)}\right),
\end{align*}
for $t\in(0,\diam(Z)^{2s})$ and $\nu$-almost all $\xi,\eta\in Z$, by \cite[Theorem 1.12]{Grigoryan-Hu-Hu:2018} and the discussion afterwards; see also \cite{Chen-Kumagai:2003}.

Next, we will use the heat kernel calculus from \cite{Hu-Zahle:2009} to obtain an estimate for the kernel of the analogue of the Bessel potential, $(1-\Delta^s)^{-1}$. More precisely: the heat kernel estimate means that $(p_t)$ satisfies the condition $(H_\Phi)$ of \cite[(4.5),(4.6)]{Hu-Zahle:2009} (with $\Phi(t)=(1+t)^{-(D+2s)}$, $\beta=2s$, $V(t)=t^D$). We aim to apply \cite[Theorem 4.3]{Hu-Zahle:2009} (with $\rho\equiv 1$, $\theta_1=D$). The assumption \cite[(4.8)]{Hu-Zahle:2009} requires verifying that
\begin{equation*}
	\int_0^{\infty} t^{D-2s-1}\Phi(t)\mathrm{d}t
	=\int_0^{\infty} \left(\frac{t}{(1+t)}\right)^D\cdot \left(\frac{1}{t(1+t)}\right)^{2s}\frac{\mathrm{d}t}{t}
	<\infty,
\end{equation*}
which holds as $D>2s$. Finally, the completely monotone function $f$ is $f(t)=t^{-1}$. Applying \cite[Theorem 4.3]{Hu-Zahle:2009}, we obtain that the positive definite operator $J_s = (1-\Delta_s)^{-1}$ (defined by functional calculus) is a kernel operator with kernel $k_s$, satisfying
\begin{align}\label{localeq:k_s-asymp-d-D-2s}
	k_s(\xi,\eta) &\asymp \frac{1}{d^{D-2s}(\xi,\eta)},\quad \xi,\eta\in Z.
\end{align}
Note that $J_s$ is injective by \cite[Theorem 3.1]{Hu-Zahle:2009}.

As the norm on $H_0^{s}(Z)$ is equivalent to $\langle (1-\Delta_s)\phi,\phi\rangle_{L^2}^{1/2}$, the dual of $H_0^s(Z)$, denoted $H_0^{-s}(Z) = H_0^s(Z)'$, is the completion of $L^2(Z)$ with respect to the (Hilbert space) norm $\langle \phi, J_s\phi\rangle_{L^2}^{1/2}$. Together with \eqref{localeq:k_s-asymp-d-D-2s}, this proves the first part of Theorem \ref{thm:potentials-in-intro}. The second part of the Theorem follows from duality and Remark \ref{rem:pi_s-restricts-to-H_0} (assuming that $\nu$ is the Hausdorff $D$-measure of $(Z,\nu)$).

\appendix

\section{A groupoid picture}\label{appx:groupoids}

Assume for simplicity that $G$ is a discrete group, acting on a locally compact second countable space $Z$. We also make a blanket assumptions that the representations considered are separable.

Denote $\mathcal{Z}= Z\times Z$ for the pair groupoid. The transformation groupoid $\mathcal{G}=G\ltimes Z$ acts by automorphisms on $\mathcal{Z}$ by $\mathcal{G}\to G\to \text{Aut}(\mathcal{Z})$. One can then define the \emph{fibred semi-direct product groupoid}, $\mathcal{G}\rtimes \mathcal{Z}$, as follows.

As a topological space, $\mathcal{G}\ltimes \mathcal{Z}$ is the fibre product
$$\mathcal{G}\ltimes \mathcal{Z}=_\text{def.}\{(g,z):t_\mathcal{Z}(z)=s_\mathcal{G}(g)\}\subset \mathcal{G}\times \mathcal{Z},$$
with space of objects $Z$, and source and target map given respectively by 
$t(g,z)=t_\mathcal{G}(g)$ and $s(g,z)=s_\mathcal{Z}(z)$. For $x\in Z$ denote $e_{x} = (x,x)\in\mathcal{Z}$, and $f_{x} = (\text{Id},x)\in\mathcal{G}$.
The maps $t$ and $s$ are continuous, and $\mathcal{G}$ and $\mathcal{Z}$ respectively identify  with subsets $\{(g, e_{s_{\mathcal{G}}(g)}):g\in\mathcal{G}\}$ and $\{(f_{t_\mathcal{Z}(z)},z):z\in\mathcal{Z}\}$.
The multiplication on $\mathcal{G}\ltimes \mathcal{Z}$ is the unique one that makes $\mathcal{G}$ and $\mathcal{Z}$ (identified as above) into subgroupoids, and satisfies 
$(g,e_{s_\mathcal{G}(g)})(f_{t_\mathcal{Z}(z)},z)=(g,z)$ when $t_{\mathcal{Z}}(z)=s_{\mathcal{G}}(g)$,
and $(f_{t_\mathcal{Z}(z)},z)(g,e_{s_\mathcal{G}(g)})=((\gamma,\gamma^{-1}t_{\mathcal{Z}}(z)),g^{-1}z)$ when $g=(\gamma,\gamma^{-1}s_{\mathcal{Z}}(z))\in \mathcal{G}$.
Observe that $\mathcal{G}\ltimes \mathcal{Z}$ has isotropy groups 
$$
G_z=\left\{\left((\gamma^{-1}z,\gamma),(\gamma^{-1}z,z)\right):\gamma\in G\right\}
$$
with $z\in Z$, all isomorphic to $G$.
In particular every representation of $\mathcal{G}\ltimes \mathcal{Z}$ induces isotropic representations of $G$.

Concretely, let $(\Pi,\mathcal{V}_Z)$  be a unitary representation of the transformation groupoid $\mathcal{G}=G\ltimes Z$. A cocycle,$c$ on $\mathcal{V}_Y$, i.e.~a representation of the pair groupoid $\mathcal{Z}\simeq Z\times Z$ is $\mathcal{G}$-equivariant if it satisfies an extra condition
$$
\Pi(b;\gamma)^{-1}c(\gamma b,\gamma a)\Pi(a;\gamma)=c(b,a)
$$
for $a,b\in Z$ and $\gamma\in G$.
Cocycles of this type correspond to representations of $\mathcal{G}\ltimes \mathcal{Z}$, and induce a class of representations of $G$ by the formula
$$
\pi_{a}(\gamma)=_\text{def.}c(a,\gamma a)\Pi(a;\gamma)=\Pi(\gamma^{-1}a;\gamma)c(\gamma^{-1}a,a)
$$
for $a\in Z$ and $g\in G$.
Observe that $c$ is unitary (respectively uniformly bounded) if and only if $\pi_{a}$ is unitary (respectively uniformly bounded).
Moreover $c(b,a)$ intertwines the representation $(\pi_{a},\mathcal{H}_a)$ with $(\pi_{b},\mathcal{H}_b)$ for all $a,b\in Z$.

On the other hand every uniformly bounded representation of $G$ on a Hilbert space (or even a Banach space) induces an uniformly bounded representation of $\mathcal{G}\ltimes \mathcal{Z}$:
\begin{prop}
Assuming that the action of $G$ on $Z$ is amenable, every uniformly bounded representation of $G$ can be seen as the isotropy representation obtained from a unitary representation of $\mathcal{G}$ and a $G$-equivariant uniformly bounded cocycle over $\mathcal{Z}$.
\end{prop}
\begin{proof}

Given a uniformly bounded representation $\pi$ of $G$ on a Hilbert space $\mathcal{H}$,
let $$T_\pi:\mathcal{H}\otimes_\text{proj}\mathcal{H}\rightarrow L^\infty(G\ltimes Z)
;\quad (v\otimes w)\mapsto T_\pi(v\otimes w)$$
be defined by $T_\pi(v\otimes w)(g,z)=(\pi(g)v,\pi(g)w)$.
Observe that $T_\pi$ is $G$-equivariant in the following sense:
$T_\rho\pi(g)\otimes\pi(g)=\gl(g)T_\rho$, where $\gl$ stands for the left translation.
If $E:\mathcal{L}^\infty(G\ltimes Z)\rightarrow\mathcal{L}^\infty(Z)$
denotes the $G$-equivariant conditional expectation --- which exists by amenability of the action of $G$ on $Z$ --- then $E\circ T_\pi:\mathcal{H}\otimes_\text{proj}\mathcal{H}\rightarrow\mathcal{L}^\infty(Z)$
is non-zero, $G$-equivariant and $B_z(v,w)=_\text{def.}E\circ T_\pi(v\otimes w)(z)$ defines a scalar product on $\mathcal{H}$ for all $z\in Z$.
Let $\mathcal{H}_z$ be the completion of $\mathcal{H}$ with respect to the scalar product $B_z$.
Then $\pi(g,z)=\pi(g):\mathcal{H}_z\rightarrow \mathcal{H}_{gz}$ is isometric and thus $(\pi|_{G\ltimes Z},(\mathcal{H}_z)_z)$ is an isometric representation of $G\ltimes Z$.

As 
$\|v\|_z\asymp_{\|\pi\|} \|v\|$ for all $z\in Z$ and $v\in \mathcal{H}$, there exists a bounded invertible operator $S_z:\mathcal{H}_z\rightarrow \mathcal{H}$ that satisfies $B_z(v,w)=(S_zv,S_zw)$ for all $z\in Z$ and $v,w\in \mathcal{H}$ with $\|S_z^{\pm1}\|\asymp\|\pi\|^{\pm1}$.
It is an exercise to prove that 
$c(a,b)=_\text{def.}S_a^{-1}S_b$ is $\mathcal{G}$-equivariant cocycle, and thus $(\pi|_{G\ltimes Z},(\mathcal{H}_z)_z,c)$ defines a representation of $\mathcal{G}\ltimes \mathcal{Z}$.
Observe that the isotropic representations  $(\pi_z)_z$ satisfy
$\pi_z(g)=S_{gz}^{-1}\pi(g)S_z$
for all $z\in Z$.
\end{proof}

\section{Some calculations}\label{appx:calculations}

In this Appendix, we spell out some calculations postponed from Section \ref{sect:outline}. Here $(Z,d)$ is a complete metric space without isolated points, and $\nu$ is an Ahlfors $D$-regular Hausdorff $D$-measure on $(Z,d)$.

\begin{calc}\label{calc:conformal-Delta}
	We verify the conformal identity for $\Delta_a^s$ from \eqref{eq:fract-Laplace-formula}, i.e.~that $\Delta_{ga}^s(g^*\phi)=\D{g}^{2s}(a)\cdot g^*(\Delta_a^s\phi)$. Here $g\in\Mob(Z)$, $a\in Z$, $\phi\in\mathcal{D}_a$, and $\xi\in Z_a$:
	\begin{align*}
		\Delta_{ga}^s(g^*\phi)(\xi) &= - \int_{Z_{ga}}\frac{g^*\phi(\eta)-g^*\phi(\xi)}{d^{D+2s}_{ga}(\eta,\xi)}\,\mathrm{d}\nu_{ga}(\eta)\\
		&= -\int_{Z_{ga}} \frac{\phi(\omega)-g^*\phi(\xi)}{d_{ga}^{D+2s}(g\omega,\xi)}\, \mathrm{d}\nu_{ga}(g\omega)\\
		&\stackrel{\eqref{eq:change-of-var-nu_a}\eqref{eq:GMV-for-d_a}}{=} -\int_{Z_{a}} \frac{\phi(\omega)-\phi(g^{-1}\xi)}{d_{a}^{D+2s}(\omega,g^{-1}\xi)}\cdot \D{g}^{D+2s}(a)\cdot \frac{\mathrm{d}\nu_{a}(\omega)}{\D{g}^{D}(a)}\\
		&= \D{g}^{2s}(a)\cdot \Delta_a^s\phi(g^{-1}\xi)
		= \D{g}^{2s}(a)\cdot g^*(\Delta_a^s\phi)(\xi).
	\end{align*}
\end{calc}

\begin{calc}\label{calc:pi_z_g_a-is-unitary}
	We verify that for an $s$-conformal family $(\mathcal{D}_a,T_a)_{a\in Z}$, $g\in\Mob(Z)$, $a\in Z$, $p\geq2$, and $z=s+it\in\mathbb{C}$,  the operator $\Pi_z(g;a):\mathcal{H}_{T_a}\to \mathcal{H}_{T_{ga}}$ is unitary. Recall that
	$\Pi_{z}(g;a)\phi = \D{g}^{\frac{D}{2}-z}(a)\, g^*\phi$.
	Here $\phi,\psi\in\mathcal{D}_a$:
	\begin{align*}
		\langle T_{ga}\Pi_{z}(g;a)\phi, \Pi_{z}(g;a)\psi\rangle_{L^2(Z_{ga})}
		&\stackrel{\eqref{eq:Pi_s_g_a-defn}}{=} \D{g}^{D-z-\overline{z}}(a) \langle T_{ga}(g^*\phi), g^*\psi \rangle_{L^2_{ga}}\\
		&= \D{g}^{D-2s}(a) \langle T_{ga}(g^*\phi), g^*\psi \rangle_{L^2_{ga}}\\
		&= \D{g}^D(a) \langle g^*(T_a\phi), g^*\psi\rangle_{L^2_{ga}}\\
		&= \D{g}^D(a) \int_{Z_ga} T_a\phi(g^{-1}\xi) \overline{\psi(g^{-1}\xi)}\,\mathrm{d}\nu_{ga}(\xi)\\
		&= \int_{Z_{ga}} T_a\phi(\omega)\overline{\psi(\omega)}\,\D{g}^{D}(a)\,\mathrm{d}\nu_{ga}(g\omega)\\
		&\stackrel{\eqref{eq:change-of-var-nu_a}}{=} \int_{Z_{a}} T_a\phi(\omega)\overline{\psi(\omega)}\,\mathrm{d}\nu_{a}(\omega)\\
		&= \langle T_a\phi, \psi\rangle_{L^2(Z_a)}.
	\end{align*}
\end{calc}

\section{A remark on tangent spaces and conformal rescaling of metric operators}
\label{appx:tangent-and-rescaling}

In this Appendix, we offer a heuristic for thinking about $(Z_a,d_a)$ as tangent spaces of $(Z,d)$, and how to interpret the conformal operators on the bundle $(Z_a,d_a)_{a\in Z}$ as a rescaling of operators on $Z$.

Throughout this section, we suppose that $(X,\rho)$ is a roughly geodesic, strongly hyperbolic metric space (cf.~\cite{Mineyev:2007,Nica-Spakula:2016} for background). Fix a basepoint $o\in X$, and let $Z=\partial X$. This setup has been explained before, cf.~Theorem \ref{thm:UB-reps-in-intro} and Section \ref{sect:almost-invariant-vectors}: denote by $\langle\cdot,\cdot\rangle_o$ the Gromov product (which continuously extends to the compactification $X\cup Z$), by $d(\xi,\eta) = \exp(-\varepsilon_0\langle\xi,\eta\rangle_o)$ (for a suitable $\varepsilon_0>0$) a metric on $Z$, by $D$ the Hausdorff dimension of $(Z,d)$, and by $\nu$ the (Ahlfors $D$-regular) measure on $Z$.

Suppose that a group $\Gamma$ acts properly discontinuously and cocompactly by isometries on $(X,\rho)$ (and hence $\Gamma<\Mob(Z)$).
We shall also need that in this situation, for $g\in\Gamma$ and $\xi\in Z$, we have $\D{g}(\xi)=\exp(-\varepsilon_0 \beta(\xi,go,o))$, where for $x\in X$ we denote $\beta(\xi,x,o) = -2\langle \xi,x\rangle_o + \rho(x,o)$ the Busemann function \cite[Section 3]{Nica-Spakula:2016}.

\subsection{Metric tangent spaces}

Fix $a\in Z$. Roughly speaking, we explain below that any fixed ball $B(z_0,R)$ in $(Z_a,d_a)$ is a (Gromov--Hausdorff) limit of rescaled small balls in $Z$ ``in the direction of $a$''.

In this subsection, let us use the notation $B_d(z,R)$ to denote (open) balls centred at $z$, with radius $R$, with respect to the metric $d$.

Suppose that we fix a sequence $(g_n)_{n\in\mathbb{N}} \subset \Gamma$ such that $g_n o \to a$ in the compactification $X\cup Z$, as $n\to\infty$. Denote $\lambda_n = \exp(\varepsilon_0\rho(g_no,o))$; then necessarily $\lambda_n\to \infty$.

Fix a basepoint $z_0\in Z_a$, and $R>0$. We now show that the balls $B_{\lambda_n d}(g_nz_0,R)$ ``converge'' to the ball $B_{d_a}(z_0,R)$ in $(Z_a,d_a)$ in the following sense: there exist functions $f_n: B_{\lambda_n d}(g_nz_0,R) \to Z_a$, such that
\begin{itemize}
	\item $f_n(g_nz_0) = z_0$ for all $n\in\mathbb{N}$,
	\item for any $\delta>0$, $f_n$ are eventually surjective onto $B_{d_a}(z_0,R-\delta)$, and
	\item the restrictions of the metrics to the balls converge uniformly to the metric on $B_{d_a}(z_0,R)$, i.e.
	 $$
	 \lim_{n\to\infty}\sup_{\xi,\eta\in B_{d_a}(z_0,R)}\{|\lambda_n d(\xi_n,\eta_n) - d_a(\xi,\eta)| :
		f_n(\xi_n)=\xi, f_n(\eta_n)=\eta \}=0.
	 $$
\end{itemize}
For $n\in\mathbb{N}$ and $\xi\in B_{\lambda_n d}(g_nz_0,R)$, we set $f_n(\xi) = g_n^{-1}\xi$. Then the above properties follow from the calculation below. Let $\xi_n,\eta_n\in B_{\lambda_n d}(g_nz_0,R)$, and denote $\xi=f_n(\xi_n)$, $\eta=f_n(\eta_n)$. Then
\begin{align*}
	\lambda_n d(\xi_n,\eta_n) &=
	\lambda_n d(g_n \xi, g_n\eta)
	\stackrel{\eqref{eq:GMV}}{=} \lambda_n \cdot \D{g_n}^{\frac12}(\xi)\cdot \D{g_n}^{\frac12}(\eta)\cdot d(\xi,\eta)\\
	&=\exp(\varepsilon_0 \rho(g_no,o))\cdot \exp(\varepsilon_0(\langle\xi,g_no\rangle_o - \rho(g_no,o)/2))\cdot \\
	&\quad \cdot \exp(\varepsilon_0(\langle\eta,g_no\rangle_o - \rho(g_no,o)/2)) \cdot d(\xi,\eta)\\
	&=\frac{d(\xi,\eta)}{\exp(-\varepsilon_0\langle\xi,g_no\rangle_o)\cdot \exp(-\varepsilon_0\langle\xi,g_no\rangle_o) }\\
	&\longrightarrow_{n\to\infty} \frac{d(\xi,\eta)}{d(\xi,a)d(\eta,a)} = d_a(\xi,\eta).
\end{align*}
This fits into the framework of weak tangent spaces from \cite{Burago-Burago-Ivanov:2001} (see also \cite{MR1616732}). One may also extend this to the framework of measured Gromov--Hausdorff convergence to include the measures $\nu$ and $\nu_a$.

\subsection{Rescaling of operators}

We continue in the setup above. Suppose we have an operator $T$ on $L^2(Z)$, with domain $\mathcal{D}(T)\subset L^2(Z)$ stable under the action of $\Gamma$. We say that an operator $(\mathcal{D}(T_a),T_a)$ on $L^2(Z_a)$ is a \emph{rescale} of $T$, if there is a sequence $(\lambda_n)_{n\in\mathbb{N}}\subset (0,\infty)$ with $\lambda_n\to\infty$ as $n\to\infty$, and a sequence $(g_n)_{n\in\mathbb{N}}\subset \Gamma$ with $g_no \to a$ in $X\cup Z$, such that
\begin{gather*}
	\phi\in\mathcal{D}(T_a) \implies g_n^*\phi\in \mathcal{D}(T)\text{ for all }n\in\mathbb{N}, \text{ and}\\
	\langle T_a\phi,\psi\rangle_{L^2(Z_a)} =  \lim_{n\to\infty} \left\langle T(\lambda_ng_n^*\phi), \lambda_n g_n^*\psi \right\rangle_{L^2(Z)},
\end{gather*}
for $\phi\in \mathcal{D}(T_a)$, $\psi\in L^2(Z_a)$.

We shall check that for $s\in(0,1)$, the fractional Laplacian \eqref{eq:fract-Laplace-formula} is a rescale of
\begin{equation*}
	\Delta^s\phi(\xi) = -\int_{Z} \frac{\phi(\eta)-\phi(\xi)}{d^{D+2s}(\xi,\eta)}\,\mathrm{d}\nu(\eta), \quad \xi\in Z,
\end{equation*}
with the domain $\mathcal{D}(\Delta^s)$ consisting of compactly supported Lipschitz functions on $Z$.
We fix any $(g_n)_{n\in\mathbb{N}}$ as above (i.e.~such that $g_no\to a$), denote $\lambda_n=\exp(\varepsilon_0(D-sp)\rho(g_no,o)/2)$.

For notational convenience, we shall write $d(\xi,x)=\exp(-\varepsilon_0\langle \xi,x\rangle_o)$ for $\xi\in Z$ and $x\in X$. Although this is not a metric, we have that $d(\xi,x)\to d(\xi,a)$ for $x\to a$ in $X\cup Z$. Similarly, we write $\mathrm{d}\nu_{x}(\xi) = d^{-2D}(x,\xi)\mathrm{d}\nu(\xi)$; these are measures on $Z$ such that $\nu_x\to \nu_a$ as $x\to a$. Finally, we write $d_{x}(\xi,\eta) = d(\xi,\eta) / (d(\xi,x)d(\eta,x))$.

As a preparatory step, we calculate for $n\in\mathbb{N}$ and $\xi,\eta\in Z$:
\begin{gather}\label{localeq:prep-for-rescale-calculation}
	\begin{aligned}
		\lambda_n^2\D{g_n}^{\frac{D}{2}-s}(\xi)\D{g_n}^{\frac{D}{2}-s}(\eta)
		&= \exp(\varepsilon_0(D-2s)\rho(g_no,o))\cdot\\
		&\qquad\cdot\exp\!\left(\varepsilon_0\!\left(\tfrac{D}{2}-s\right)\!\left(2\langle\xi,g_no\rangle -\rho(g_no,o)\right)\right) \cdot\\ &\qquad\cdot\exp\!\left(\varepsilon_0\!\left(\tfrac{D}{2}-s\right)\!\left(2\langle\eta,g_no\rangle -\rho(g_no,o)\right) \right)\\
		&=d^{-D+2s}(\xi,g_no)\,d^{-D+2s}(\eta,g_no).
	\end{aligned}
\end{gather}
Now we calculate
\begin{multline*}
	\left\langle \Delta^s(\lambda_ng_n^*\phi), \lambda_n g_n^*\psi \right\rangle_{L^2(Z)}\\
\begin{aligned}
		&= \frac{\lambda_n^2}{2} \iint_{Z^2}\frac{(\phi(g_n^{-1}\xi)-\phi(g_n^{-1}\eta))\overline{(\psi(g_n^{-1}\xi)-\psi(g_n^{-1}\eta))}}{d^{D+2s}(\xi,\eta)} \mathrm{d}\nu(\xi)\mathrm{d}\nu(\eta) \\
		&\stackrel{\eqref{eq:change-of-var},\eqref{eq:GMV}}{=} \frac{\lambda_n^2}{2}
		\iint_{Z^2}\frac{(\phi(\xi)-\phi(\eta))\overline{(\psi(\xi)-\psi(\eta))}}{d^{D+2s}(\xi,\eta)}
		\D{g_n}^{\frac{D}{2}-s}(\xi) \D{g_n}^{\frac{D}{2}-s}(\eta)\,\mathrm{d}\nu(\xi)\mathrm{d}\nu(\eta) \\		
		&\stackrel{\eqref{localeq:prep-for-rescale-calculation}}{=} \frac{1}{2}
		\iint_{Z^2}\frac{(\phi(\xi)-\phi(\eta))\overline{(\psi(\xi)-\psi(\eta))}}{d_{g_no}^{D+2s}(\xi,\eta)}
		\mathrm{d}\nu_{g_no}(\xi)\mathrm{d}\nu_{g_no}(\eta)\\
		&\longrightarrow \langle \Delta_a^s\phi,\psi\rangle_{L^2(Z_a)},
\end{aligned}
\end{multline*}

Note that a similar calculation works more generally for the fractional $p$-Laplacian ($p\geq 1$), which is a $(p-1)$-homogeneous nonlinear operator $\Delta_p^2: W^{s,p}(Z)\to W^{s,p}(Z)'$, defined by
$$
\langle\Delta_p^s(u),v\rangle_{(W^{s,p})',W^{s,p}}=\frac{1}{2} \iint\frac{|u(\xi)-u(\eta)|^{p-2}(u(\xi)-u(\eta))\ol{(v(\xi)-v(\eta))}}{ d^{D+sp}(\xi,\eta)}\mathrm{d}\nu(\xi) \mathrm{d}\nu(\eta),
$$
for $u,v\in W^{s,p}(Z)$.

\nocite{*}
\bibliographystyle{plain}
\bibliography{sobolev-cayley}
\end{document}